\documentclass[a4paper, reqno]{amsart}

\usepackage{quotes}
\usepackage{amsmath,amsthm,amssymb,amsfonts,mathtools}
\usepackage{bbm}
\usepackage{enumerate}
\usepackage{hyperref}
\usepackage[dvipsnames]{xcolor}
\usepackage{algorithm}
\usepackage[]{algpseudocodex}

\usepackage{relsize}
\usepackage{exscale}

\theoremstyle{plain}
\newtheorem{theorem}{Theorem}[section]
\newtheorem{proposition}[theorem]{Proposition}
\newtheorem{lemma}[theorem]{Lemma}

\theoremstyle{remark}
\newtheorem{rem}[theorem]{Remark}

\theoremstyle{definition}

\newtheorem{assm}{Assumption}

\newcommand{\E}[2][]{\mathbb{E}_{#1} \left[ #2 \right]}
\newcommand{\abs}[1]{\left| #1 \right|}

\newcommand{\Prob}[2][]{\mathbb{P}_{#1} \! \left( #2 \right)}

\newcommand{\ind}[1]{\mathbbm{1}_{#1}}
\newcommand{\rb}[1]{\!\left( #1 \right)}
\newcommand{\cb}[1]{\! \left\{ #1 \right\}}
\newcommand{\sqb}[1]{\! \left[ #1 \right]}

\newcommand{\cost}[1]{\mathrm{Cost}\!\left(#1\right)}

\DeclareMathOperator{\tr}{tr}

\newcommand{\bN}{\mathbb{N}}
\newcommand{\bP}{\mathbb{P}}
\newcommand{\bR}{\mathbb{R}}

\newcommand{\cF}{\mathcal{F}}
\newcommand{\cI}{\mathcal{I}}
\newcommand{\cL}{\mathcal{L}}
\newcommand{\cO}{\mathcal{O}}
\newcommand{\cT}{\mathcal{T}}

\newcommand{\Dt}{\Delta t}
\newcommand{\mesh}[1]{\cT^{#1}}

\newcommand{\barX}{\overline{X}}
\newcommand{\tildeW}{\widetilde{W}}
\newcommand{\barW}{\overline{W}}

\newcommand{\rdt}{\mathrm{d}t}
\newcommand{\rdW}{\mathrm{d}W}
\newcommand{\rdX}{\mathrm{d}X}

\newcommand{\reach}{\mathrm{reach}}

\title[Higher-order adaptive methods for exit times]{Higher-order adaptive methods for exit times of It\^o diffusions} 
\author[H. Hoel]{H{\aa}kon Hoel${}^*$}
\thanks{$^*$Corresponding author: H. Hoel (haakonah@math.uio.no)}
\address{Department of Mathematics, University of Oslo, Oslo, Norway} 
\email{haakonah@math.uio.no}

\author[S. Ragunathan]{Sankarasubramanian Ragunathan}
\address{Chair of Mathematics for Uncertainty Quantification,
RWTH Aachen University, Aachen, Germany}
\email{sankarasubramanian.ragunathan@rwth-aachen.de}

\keywords{exit time, SDE, adaptive time-stepping, Itô--Taylor scheme, strong approximation, Feynman--Kac}

\subjclass[2020]{60H10, 60H30, 60H35, 65C30}

\begin{document}

    \begin{abstract}
            
        We construct a higher-order adaptive method for strong approximations of 
        exit times of Itô stochastic differential equations (SDE). The method employs a strong Itô--Taylor scheme for simulating SDE paths, and adaptively decreases the step size in the numerical integration as the solution approaches the boundary of the domain. These techniques turn out to complement each other nicely: adaptive time-stepping improves the accuracy of the exit time by reducing the magnitude of the overshoot of the numerical solution when it exits the domain, and higher-order 
        schemes improve the approximation of the state of the diffusion process. 
        
        We present two versions of the higher-order adaptive method. The first one uses 
        the Milstein scheme as numerical integrator and two step sizes for adaptive time-stepping: $h$ when far away from the boundary and $h^2$ when close to the boundary. 
        The second method is an extension of the first one using the strong It\^o--Taylor scheme of order 1.5 as numerical integrator and three step sizes for adaptive time-stepping. For any $\xi>0$, we prove that the strong error is $\cO(h^{1-\xi})$ and $\cO(h^{3/2-\xi})$ for the first and second method, respectively, and the expected computational cost for both methods is $\cO(h^{-1} \log(h^{-1}))$. 
        Theoretical results are supported by numerical examples, and we discuss the potential for extensions that improve the strong convergence rate even further.  
    \end{abstract}

    \maketitle

    \section{Introduction}\label{sec:introduction}

    For a bounded non-empty domain $D \subset \bR^d$ and a $d-$dimensional It\^o diffusion $(X_t)_{t\ge 0}$ 
    with $X(0) \in D$, the goal of this work is to construct efficient higher-order adaptive 
    numerical methods for strong approximations of the exit time 
    \begin{equation}\label{eq:exit-time}
    \tau := \inf\{t \ge 0 \mid X(t) \notin D \} \wedge T,
    \end{equation}
    where $T>0$ is given. The dynamics of the diffusion process is  
    governed by the autonomous It\^o stochastic differential equation (SDE)
    \begin{equation}\label{eq:SDE}
        \begin{split}
            \rdX &= a(X)\rdt+ b(X)\rdW\\
            X(0) &= x_0,
        \end{split}
    \end{equation}
    where $a \colon \mathbb{R}^{d} \to \mathbb{R}^{d}$ and $b \colon \mathbb{R}^{d} \to \mathbb{R}^{d \times m}$, $x_0$ is a deterministic 
    point in $D$, and $W$ is an $m$-dimensional standard Wiener process on a filtered probability space $(\Omega, \cF, (\cF_t)_{t \ge 0}, \bP)$. Further details on the regularity of the domain $D$ and the coefficients of the SDE are provided in Section~\ref{sec:notation-and-methods}.
   
    Our method employs a strong Itô--Taylor scheme for simulating SDE paths and carefully decreases the step size in the numerical integration as the solution approaches the boundary of the domain.  We present two versions: the order 1 method and the order 1.5 method. The order 1 method uses the Milstein scheme as numerical integrator and two step sizes for the adaptive time-stepping depending on the proximity of the state the boundary of $D$: a larger time-step when far away and a smaller time-step when close to $\partial D$. The order 1.5 method uses the strong It\^o--Taylor scheme of order 1.5 as numerical integrator and, as an extension of the previous method, three different step sizes for the adaptive time-stepping. 
    When adaptive time-stepping is properly balanced against the order of the scheme
    the following results hold for any $\xi>0$: the 
    order 1 method achieves the strong convergence rate $\cO(h^{1-\xi})$, cf. Theorem~\ref{thm:error-order1}; and the order 1.5 method achieves the strong convergence rate $\cO(h^{3/2-\xi})$, cf.~ Theorem~\ref{thm:error-order1.5}.
    Theorems~\ref{thm:cost-order1} and~\ref{thm:cost-order1.5} further show that
    both of the higher-order adaptive methods have an expected computational cost of $\cO(h^{-1} \log(h^{-1}))$.
   
    Exit times describe critical times in controlled dynamical systems, and they have applications in pricing of barrier options~\cite{alsmeyer1994markov}
    and also in American options~\cite{bayer2020pricing}, as the latter may be formulated as a control problem where exit times determine when to execute the option.
    They also appear in physics, for instance when studying the transition time between pesudo-stable states of molecular systems in the canonical ensemble~\cite{weinan2021applied}. 
    
    The mean exit time problem may be solved by the Monte Carlo approach of direct simulation of SDE paths or by solving the associated partial differential equation; the Feynman--Kac equation~\cite{friedmanVol1}. In high-dimensional state space, the former approach is more tractable than the latter, due to the curse of dimensionality. 
    A Monte Carlo method using the Euler--Maruyama scheme with a uniform timestep to simulated SDE paths was shown to produce the weak convergence rate 1/2 for approximating the mean exit time in~\cite{Gobet2000-rf}. An improved weak order 1 method was achieved  by reducing the overshoot error through careful shifting of the boundary, cf.~\cite{broadie1997continuity,Gobet2001-pr,gobet_exact_2004}. The weak order 1 method was extended to problems with time-dependent boundaries in~\cite{gobet_stopped_2010}, and~\cite{bouchard_first_2017} showed that in $L^1(\Omega)$-norm, the Euler--Maruyama scheme has order 1/2 convergence rate. The contributions of this work may be viewed as extensions of this strong convergence study to higher-order It\^o--Taylor schemes with adaptive time-stepping. 
    
    Multilevel Monte Carlo (MLMC) methods for mean exit times have been developed in~\cite{higham_mean_2013,giles2018multilevel}. Using the Euler--Maruyama scheme with boundary shifting and conditional sampling of paths near the boundary, the method~\cite{giles2018multilevel} can reach the mean-square approximation error $\cO(\epsilon^2)$ at the near-optimal computational cost $\cO(\epsilon^{-2} \log(\epsilon)^2 )$. Since our methods efficiently approximate the strong exit time and admit a pairwise coupling of simulations on different resolutions, they are also suitable for MLMC. See Section~\ref{sec:conclusion}, for an outline of the extension to MLMC. 
    
    Adaptive methods for SDE with discontinuous drift coefficients have been considered in~\cite{neuenkirch2019adaptive,yaroslavtseva2022adaptive}.
    Therein, the discontinuity regions of the drift coefficient are associated to hypersurfaces, and when one is close to such hypersurfaces, a small time-step is used, and otherwise a large time-step is used. This adaptive-time-stepping approach is similar to ours, as can be seen by viewing the boundary of $D$
    in our problem as a hypersurface, but the problem formulations differ, and our method is more general in the sense that it admits a sequence of time-step sizes and also higher-order numerical integrators. 
    For numerical testing of strong convergence rates, we have employed the idea of pairwise coupling of non-nested adaptive simulations of SDE developed in~\cite{gilesnonnested2016}, which is an approach that also could be useful for combining our method with MLMC in the future. See also~\cite{merleProhl2022} for a recent contribution on a posteriori adaptive methods for weak approximations of exit times and states of SDE, and
    ~\cite{hoel2014implementation,fang2020adaptive,katsiolides2018multilevel} for other partly state-dependent adaptive MLMC methods for weak approximations of SDE in finite- and infinite-time settings.

    As an alternative to simulating SDE paths, the Walk-On-Sphere (WoS) scheme~\cite{muller1956some} was devised to compute the solution to the Laplace equation on a bounded domain $D$ using Monte Carlo techniques. It was later extended to the walk on moving spheres method~\cite{deaconu2017walk} for simulating the exit time and position of a Wiener processes starting inside a bounded domain. 
    WoS requires closed-form expression or a tractable approximation of exit time distributions of spheres so that the exit time of the process can be generated from walking on spheres. But for more generic SDE, such closed-form expressions or tractable approximations do not exist, limiting the applicability of WoS.

    The rest of this paper is organized as follows: Section~\ref{sec:notation-and-methods} presents two versions of the higher-order adaptive method and states the main theoretical results on convergence and computational cost. Section~\ref{sec:theory} contains proofs of the theoretical results. Section~\ref{sec:numerics} presents a collection of numerical examples supporting our theoretical results. Lastly, we summarize our findings and discuss interesting open questions in Section~\ref{sec:conclusion}.

    \section{Notation and the adaptive numerical methods}\label{sec:notation-and-methods}
    In this section, we first introduce necessary notation and assumptions on the SDE coefficients and the domain $D$, and relate the exit time problem to the Feynman--Kac partial differential equation (PDE).  Thereafter, the higher-order adaptive methods of order 1 and 1.5 are respectively described in Sections~\ref{subsec:order_1_method} and~\ref{subsec:order_15_method}, with convergence and computational cost results. 
    
    \subsection{Notation and assumptions}\label{subsec:notAndAssumptions}
    
    For any $x,y \in \bR^d$, the Euclidean distance is denoted by $d(x,y) := |x-y|$, and 
    for any non-empty sets $A,B \subset \bR^d$ we define 
    \[
        d(A,B) := \inf_{(x,y) \in A \times B} d(x,y).
    \]  
    The following assumptions on the domain of the diffusion process will be needed to 
    relate the exit-time problem to a sufficiently smooth solution of a Feynman--Kac PDE.
    \begin{assm}\label{assm:domain}
          The domain $D \subset \bR^d$ is non-empty and bounded, and the boundary $\partial D$ is $C^{4}$ continuous. 
    \end{assm}
    Let $n_D :\partial D \to \bR^d$ denote the outward-pointing unit normal and for any $r>0$, let 
    \[
    D_r := \{ x \in \bR^d \mid  d(x,D) < r \}
    \]
    Due to the regularity of the boundary, it holds that $n_D \in C^{3}(\partial D,\bR^d)$ and $D$ satisfies the uniform ball property: there exists an $R>0$ such that for every $x \in \bR^n$ such that $d(x,\partial D)< R$, there is a unique nearest point $y$ on the boundary $\partial D$, and it holds that $x = y \pm d(x,y) n_D(y)$, where the sign $\pm$ depends on whether $x$ belongs to the exterior or interior of $D$. (That is, for every $x$ sufficiently close to the boundary, there is a unique point on $\partial D$ satisfying that
    $y = \arg \min_{z \in \partial D} d(x,z)$.) We refer to~\cite[Chapter 14.6]{gilbarg2015elliptic} for further details on the regularity of $\partial D$ and its unit normal $n_D$, and to~\cite[Theorem 1.9]{dalphin2014some} for the uniform ball property. For any domain $D$ satisfying the uniform ball property, let $\mathrm{reach}(D)> 0$ denote the supremum of all $R>0$ such that the property is satisfied. 
    The uniform ball property is equivalent to satisfying both the uniform exterior sphere- and interior shere condition.
    
    For any $r\in (0, \mathrm{reach}(D))$, the boundary of $D_r$ can be expressed by 
    \[
     \partial D_r = \bigcup_{y \in \partial D}  \, y + r n_D(y) ,
    \]
    and, since the mapping 
    \begin{equation}\label{eq:diffeomorphism}
    \partial D \ni y \mapsto \pi_r(y) = y + r n_D(y) \in \partial D_r
    \end{equation}
    is a $C^3$-diffeomorphism and $\partial D$ is $C^4$, it follows that $\partial D_r$ is at least $C^3$. 
    This implies that $D_r$ also satisfies the uniform ball property and since any point $z \in \partial D_r$
    can be expressed as $z= y + r n_D(y)$ for a unique $y \in \partial D$ and indeed $n_D(y) = n_{D_r}(z)$, it follows that
    the reach of $D_r$ is bounded from below by $\mathrm{reach}(D)-r$, which we state for later reference:
    \begin{equation}
     \mathrm{reach}(D_r) \ge \mathrm{reach}(D) -r \qquad \forall r \in [0,\reach(D))   
    \end{equation}
    A similar construction applies to the boundary of the subdomain
    $D_{-r} := \{x \in D| d(x, \partial D)>r \}$, where one can show that $\partial D_{-r}$ is at least $C^3$ for all $r \in (0,\mathrm{reach}(D))$.

    For any integer $k\ge 1$, let $C^{k}_b(\bR^d)$ denote the set of scalar-valued functions with continuous and uniformly bounded partial derivatives up to and including order $k$.
    We make the following assumptions on the coefficients of the SDE~\eqref{eq:SDE}.
    \begin{assm}\label{assm:coeffs}
        \hfill
        
        \begin{enumerate}[(\ref{assm:coeffs}.1)]
    
            \item For the SDE coefficients
            \[
            a(x) = \begin{bmatrix} 
            a_1(x) \\ \vdots \\ a_d(x) \end{bmatrix}
            \qquad \text{and} \qquad 
            b(x) = \begin{bmatrix} b_{11}(x) & \cdots & b_{1m}(x)\\
            \vdots & & \vdots \\
            b_{d1}(x)& \cdots & b_{dm}(x)
            \end{bmatrix}
            \]
            it holds that $a_i, b_{ij} \in C_b^{3}(\bR^d)$ for all $i \in \{1,\ldots,d\}$ and $j \in \{1,\ldots, m\}$.

            \item (Uniform ellipticity) For some $\bar R_D \in (0,\reach(D)/2)$, there exists a constant $\hat c_b>0$ such that
             \[
            \hat c_{b} \abs{\xi}^{2} \le \xi^{\top} (b(x) b^{\top}(x)) \xi \qquad \forall (x, \xi) \in  D_{\bar R_D} \times \bR^d. 
            \]
            
            \item There exists a constant $\widehat C_{b} > 0$ such that 
            \[ 
            \xi^{\top} \rb{b(x) b^{\top}(x)} \xi \le \widehat C_{b} \abs{\xi}^{2}, \qquad \forall (x,\xi)  \in \mathbb{R}^{d}\times \mathbb{R}^{d}. 
            \]
            
            
        \end{enumerate}
    \end{assm}
    \begin{rem}
        Assumption~\ref{assm:coeffs}.1 ensures well-posedness of 
        strong exact and numerical solutions of the SDE~\eqref{eq:SDE} and is sufficient to obtain the sought regularity for the solution of the Feynman--Kac PDE in Propositions~\ref{prop:fk1} and~\ref{prop:fk2}.

        \ref{assm:coeffs}.2 ensures that the diffusion process will, loosely speaking, eventually exit the domain, and it is used to obtain well-posedness of the related Feynman--Kac PDE in Propositions~\ref{prop:fk1} and~\ref{prop:fk2}.
        \ref{assm:coeffs}.3 is introduced to bound the magnitude of the overshoot of numerical paths of the diffusion process when they exit $D$.
     \end{rem}
     To simplify technical arguments, we will prove convergence results for our numerical methods under Assumption~\ref{assm:coeffs}, but the assumption can be relaxed considerably:
    \begin{rem}
        Since the quantity we seek to compute, $\tau$, only depends on the dynamics of the diffusion process until it exits $D$, Assumption~\ref{assm:coeffs}.1 can be relaxed to:
        \[
        \textbf{(B$'.1$)} \qquad a_i, b_{ij} \in C^{3}(\bR^d) \quad \text{for all} \quad  i \in \{1,\ldots,d\} \quad \text{and} \quad  j \in \{1,\ldots, m\}. 
        \]
        
        This is because any $C^{3}$-redefinition of all coefficients on $\overline D^C$ such that $a_i, b_{ij} \in C^{3}_b(\bR^d)$ will not change the exit time of the resulting diffusion process. 
        
        For instance, if 
        \[
        \max (|a_i(x)|, |\partial_{x_j} a_i(x)|,  |\partial_{x_jx_k} a_i(x)|, |\partial_{x_jx_k x_\ell} a_i(x)| )\le C(1+ |x|^n),
        \] 
        for some $n\ge 1$, then the redefinition 
        $\tilde a_i(x) = a_i(x) \exp(- d(x, \overline D^C)^4)$ satisfies 
        $\tilde a_i \in C_b^{3}(\bR^d)$.

        For any $R>0$, Assumption~\ref{assm:coeffs}.3 can be relaxed to: There exists a $\widehat C_b>0$ such that 
        \[ 
        \mspace{-120mu} 
        \textbf{(B$'.3$)} \qquad \xi^{\top} \rb{b(x) b^{\top}(x)} \xi \le \widehat C_{b} \abs{\xi}^{2}, \qquad \forall x \in \overline D_R, \; \xi \in \mathbb{R}^{d}. 
        \]
        The relaxation is achieved through a $C^{3}$-redefinition of $b$ on the exterior of $\overline D$ satisfying that
        \[ 
        \sup_{x \in \bR^d} \|b(x) b^{\top}(x)\|_2  = \max_{x \in \overline D_R} \|b(x) b^{\top}(x) \|_2 , 
        \]
        where $\| \cdot \|_2$ denotes the matrix $2$-norm.
    \end{rem}

    \subsection{Mean exit times and Feynman--Kac}
    Recalling that the exit time was defined by $\tau = \inf\{t \ge 0 \mid X(t) \notin D \} \wedge T$,
    we extend this notation to the time-adjusted exit time of a path going through the 
    point $(t,x)\in [0,T] \times \bR^d$:
    \begin{equation*}
        \tau^{t,x} \coloneqq \Big( \min\left\{s \geq t \; \big| \; X(s) \notin D \text{ and } X(t) = x\right\} \, - \, t \, \Big)\; \bigwedge \; \big(T - t\big) \, . 
    \end{equation*}
    Recalling that $X(0) = x_0$ for the diffusion process~\eqref{eq:SDE}, we note that $\tau = \tau^{0,x_0}$. Under sufficient regularity, the mean (time-adjusted) exit time 
    \begin{equation}\label{eq:mean-exit-time}
        u(t, x) \coloneqq \E{\tau^{t, x}}, \qquad (t,x) \in \sqb{0, T} \times \mathcal{D}
    \end{equation}
    is the unique solution of the following parabolic PDE:
    \begin{proposition}[Feynman--Kac]\label{prop:fk1}
    If Assumptions~\ref{assm:domain},~\ref{assm:coeffs}.1 and~\ref{assm:coeffs}.2 hold, then 
    the mean exit time~\eqref{eq:mean-exit-time} is the unique solution of the Feynman--Kac PDE
     \begin{align*}
         \partial_t u &=  - a \cdot \nabla u - \frac{1}{2} \tr\rb{bb^{\top} \nabla^{2} u} - 1   && \text{in}  \quad (0,T) \times D, \\
        u  &= 0  && \text{on} \quad  ([0,T) \times \partial D) \cup (\{T\} \times D) \, .
        \end{align*}
    Moreover, $u \in C^{1,2}([0,T]\times D) \cap C^{0,0}([0,T] \times \overline D)$.
    \end{proposition}
    The result is a direct consequence of Proposition 1 in~\cite{gobet_stopped_2010}.

    To bound the overshoot of the exit time of numerical solutions, it is useful to study 
    time-adjusted exit times on domains $D_r$ for any $r \in [0,\bar R_D]$, where $\bar R_D$ 
    is defined in Assumption~\ref{assm:coeffs}.2. For any $(x,t) \in [0,T] \times D_r$, let 
    \[
     \tau^{t,x}_r \coloneqq \Big( \min\left\{s \geq t \; \big| \; X(s) \notin D_r \text{ and } X(t) = x\right\} \, - \, t \, \Big)\; \bigwedge \; \big(T - t\big) \, , 
    \]
    Similarly as for the exit time for the domain $D$, we introduce the shorthand $\tau_r := \tau_r^{0,x_0}$. The mean exit time for the enlarged domain is defined by 
    \begin{equation}\label{eq:mean-exit-time-enlarged-domain}
    u_r(t,x) := \E{\tau^{t,x}_r} \qquad (t,x) \in [0,T] \times \bR^d.
    \end{equation}
    The function $u_r(t,x)$ is also the unique solution of a Feynman--Kac equation:
    \begin{proposition}[Feynman--Kac on enlarged domains] \label{prop:fk2}
    Let Assumptions~\ref{assm:domain},~\ref{assm:coeffs}.1 and~\ref{assm:coeffs}.2 hold
    and let $r \in [0,\bar R_D]$. Then the mean exit time~\eqref{eq:mean-exit-time-enlarged-domain} for the enlarged domain $D_r$ is the unique solution of the Feynman--Kac PDE
     \begin{align*}
         \partial_t u_r &=  - a \cdot \nabla u_r - \frac{1}{2} \tr\rb{bb^{\top} \nabla^{2} u_r} - 1   && \text{in}  \quad (0,T) \times D_r, \\
        u_r  &= 0  && \text{on} \quad  ([0,T) \times \partial D_r) \cup (\{T\} \times D_r) \, ,
        \end{align*}
    and $u_r \in C^{1,2}([0,T]\times D_r) \cap C^{0,0}([0,T] \times \overline D_r)$.
    
    Moreover, 
    \begin{equation}\label{eq:ur-larger-u}
    u_r(t,x) \ge u(t,x) \qquad (t,x) \in [0,T] \times D.
    \end{equation}
    and there exists a uniform constant $L >0$ for all $r \in [0,\bar R_D]$ such that 
    \begin{equation}\label{eq:lipschitz-u}
        | u_r(t,x) - u_r(t,y) | \le L |x - y|  \quad \forall (t,x,y) \in [0,T] \times  D_r \times \partial D_r \, . 
    \end{equation}
    \end{proposition}
    
    \begin{proof}
    Assumption~\ref{assm:domain} and the argument right below the assumption shows that 
    $\partial D_r$ is $C^3$. By~\cite[Proposition 1]{gobet_stopped_2010}, we then conclude that $u_r$ is the unique solution with the stated regularity. Inequality~\eqref{eq:ur-larger-u} follows from the observation
    that 
    \[
    \tau_r^{t,x} \ge \tau^{t,x} \qquad \forall (t,x) \in [0,T] \times \bR^d.
    \]
    
    For the last inequality, we note that for any $r \in [0,R_D]$, the non-truncated mean exit time 
    \[
    \bar u_r(x) := \E{\bar \tau_r^x}, 
    \]
    with $\bar \tau_r^x := \min\{t \ge 0 \mid X(s) \notin D_r \; \& \; X(0) = x \}$
    is the unique solution of the strongly elliptic PDE
    \begin{equation}\label{eq:familyPDE}
    \begin{split}
     a \cdot \nabla \bar u_r + \frac{1}{2} \tr\rb{bb^{\top} \nabla^{2} \bar u_r}  + 1 & = 0 \qquad \text{in} \quad D_r \\
     \bar u_r &= 0 \qquad \text{on} \quad \partial D_r,
        \end{split}
    \end{equation}
    and $\bar u_r \in C^2(\overline D_r)$, cf.~\cite[Theorem 6.5.1]{friedmanVol1} and~\cite[Theorem 6.14]{gilbarg2015elliptic}. 
    Since $u_r(t,y) = \bar u_r(y) = 0$ for all $y \in \partial D_r$ and  
    \[
    |u_r(t,x) - u_r(t,y)| = u_r(t,x) - \underbrace{u_r(t,y)}_{=0} \stackrel{\eqref{eq:ur-larger-u}}{\le} \bar u_r(x) - \underbrace{\bar u_r(y)}_{=0} = |\bar u_r(x) - \bar u_r(y)|,
    \]
    and it remains to bound the term in the last equation.   
    By the connection of modulus of continuity~\cite[Section~14.5]{gilbarg2015elliptic} with 
    gradient bounds on the boundary, \cite[Theorem 14.1]{gilbarg2015elliptic} applied to 
    the PDE~\eqref{eq:familyPDE} yields
    \[
        |\bar u_r(x) - \bar u_r(y)| \le \underbrace{\mu e^{\chi_r M_r}}_{=:L_r} |x - y|  \qquad \forall (x,y) \in D_r \times \partial D_r,
    \]
    where $M_r = \sup_{x \in \overline D_r} |\bar u_r(x)|$, 
    \[
    \mu = \frac{1+ \widehat C_b + \sup_{x \in D_{\bar R_D}} |a(x)|}{\hat c_b}, \quad \text{and} \quad   
    \chi_r = \left(1 + \frac{d-1}{\reach(D_r)} \right) \mu.
    \]
    For any $ 0 \le r< s\le \bar R_D$, it holds that $u_s \ge u_r \ge 0$ on $\overline D_r$, 
    which implies that $M_r \le M_{\bar R_D}$. 
    And from Section~\ref{subsec:notAndAssumptions} and Assumption~\ref{assm:coeffs}.2, we know that 
    that $\reach(D_r) \ge \reach(D)/2>0$ for all $r \in [0,\bar R_D]$, which implies that  
    \[
    \chi_r \le \left(1 + \frac{2(d-1)}{\reach(D)} \right) \mu =: \bar \chi \quad \forall r \in [0, \bar R_D].
    \]  
    We conclude that  $L_r \le \mu \exp( \bar \chi M_{\bar R_D}) =: L < \infty$ for all $r \in [0, \bar R_d]$.
    %
    %

\end{proof}

    \begin{rem}
    Existence, uniqueness and regularity of the solution to the elliptic PDE can be shown to hold under weaker regularity assumptions on coefficients, domain and boundary values \cite[Theorems 6.1, 6.2, 6.3]{gilbarg1980intermediate}.
    \end{rem}

\subsection[]{Higher-order adaptive numerical methods}\label{subsec:Ito-Taylor}
    
    Let $\gamma \in \{1,3/2\}$ denote the order of the strong It\^o--Taylor method used in the numerical integration of the SDE~\eqref{eq:SDE}, cf.~\cite[Chapter 10]{kloeden1992stochastic}.
    In abstract form, our numerical method simulates the SDE~\eqref{eq:SDE} by the scheme
    \begin{equation}\label{eq:ito-taylor-scheme}
    \barX(t_{n+1}) = \Psi_{\gamma}(\barX(t_n), \Delta t_n), \qquad n=0,1,\ldots  
    \end{equation}
    where $\barX(0) = x_0$ and $\Psi_\gamma: \bR^d \times [0,\infty) \times \Omega \to \bR^d$ denotes the higher-order It\^o--Taylor scheme. The timestep $\Delta t_n = \Delta t(\barX(t_n))$ is adaptively chosen as a function of the current state $\barX(t_n)$ so that the step size is small when $\barX(t_n)$ is near the boundary $\partial D$, and larger otherwise. Both 
    the integrator $\Psi_\gamma$ and the adaptive time-stepping depend on the order of $\gamma$, see Sections~\ref{subsec:order_1_method} and~\ref{subsec:order_15_method} below for further details. The purpose of the adaptive strategy is to reduce the magnitude of the overshoot when the numerical solution exits $D$, and the stochastic mesh
    \[
    t_0 = 0, \quad \text{and} \quad t_{n+1} = t_n + \Delta t_n \qquad n =0,1,\ldots
    \]
   is a sequence of stopping times with $t_{n+1}$ being  $\cF_{t_n}-$measurable for all $n=0,1,\ldots$.
   The exit time of the numerical method is defined as the first time $\barX(t_n)$ exits the domain $D$:
    \begin{equation}\label{eq:exit-time-numerical}
    \nu := \{t_n\ge 0 \mid \barX(t_n) \notin D\} \wedge T, 
    \end{equation}
    and we will use the following notation for the time mesh 
    of the numerical solution up to or just beyond $T$:
    \begin{equation}
        \label{eq:adpt_mesh}
        \mesh{\Dt} := \{t_0, t_1, \ldots, t_{N}\}
    \quad \text{where} \quad 
    N := \min\{n \in \bN \; | \; t_n \ge T \}.
    \end{equation}
    For later analysis, the domain of definition for the numerical solution is extended to a piecewise 
    constant solution over continuous time by
    \[
    \barX(t) := \barX(t_n) \qquad \text{for} \quad t \in [t_n,t_{n+1}) 
    \quad \text{and} \quad n \in \{0,1,\ldots, N-1\}.  
    \]
    Sections~\ref{subsec:order_1_method} and~\ref{subsec:order_15_method} below present the details for the order 1 and order 1.5 adaptive methods, respectively.

    \subsection{Order 1 method}
    \label{subsec:order_1_method}
    The $i$-th component of the strong It\^o--Taylor scheme~\eqref{eq:ito-taylor-scheme} of 
    order $\gamma=1$ is given by 
    \begin{equation}\label{eq:order1-scheme}
        \begin{split}
            \Big(\Psi_{1}(\barX(t_{n}), \Delta t_n ) \Big)_i  \coloneqq \barX_{i}(t_{n}) &+ a_{i}(\barX(t_{n})) \Delta t_n + \sum_{j=1}^{m} b_{ij}(\barX(t_{n})) \Delta W^j_n   \\ &+ \sum_{j_{1}=1}^m \sum_{j_2=1}^{m} \mathcal{L}^{j_{1}} b_{ij_{2}}(\barX(t_{n})) \mathcal{I}_{(j_{1}, j_{2})},
        \end{split}
    \end{equation}
    for $i \in \{1,\ldots, d\}$,
    where we have introduced the shorthand $\Delta W^j_n = W^j(t_{n+1}) - W^j(t_n)$
    for the $n$-th increment of the $j$-th component of the Wiener process,
    the operator
    \begin{equation}\label{eq:Lj-operator}
    \cL^j := \sum_{j=1}^d b_{ij} \partial_{x_i}
    \end{equation}
    and the iterated It\^o integral
    \begin{equation}\label{eq:iterated-ito}
         \mathcal{I}_{(j_{1}, j_2)} := \int_{t_{n}}^{t_{n+1}} \int_{t_{n}}^{s_{2}}\rdW^{j_{1}}(s_1)\rdW^{j_{2}}(s_2)
    \end{equation}
    of the components $(j_1,j_2) \in \{1,\ldots,m\} \times \{1,\ldots,m\}$. 
    
    \begin{rem}
    The iterated integrals $\cI_{(j_1, j_2)}$ do not have a closed-form expression when $j_1\neq j_2$, and numerical approximations of such terms impose a substantial 
    cost to each iteration of $\Psi_1$ in~\eqref{eq:order1-scheme}. The cost of evaluating
    $\Psi_1$ reduces to $\cO(1)$ in settings when the off-diagonal terms of $\cI_{(j_1, j_2)}$ cancel; for instance, when the \emph{first commutativity condition} holds~\cite[equation (10.3.13)]{kloeden1992stochastic}:
    \begin{align}
        \label{eq:first_comm}
          \cL^{j_1}  b_{ij_{2}} = \cL^{j_2} \partial_{x_i} b_{ij_{1}}, \qquad 
        \forall \big( \; j_{1}, j_{2} \in \cb{1, \ldots, m} \text{ and } i \in \{1, \ldots, d\} \; \big).
    \end{align}
    \end{rem}
    
    The size of $\Delta t_n$ is determined adaptively by the state of the numerical solution. A small step size is employed when $\barX(t_n)$ is close to the boundary $\partial D$, to reduce the magnitude of the overshoot of an exit of the domain and the likelihood of the numerical solution not capturing a true exit of the domain; and a larger step size is employed when $\barX(t_n)$ is farther away from the boundary. For any $b>a\ge0$, we introduce the following notation to describe the distance from the boundary:
    \begin{align}
        \label{eq:crit_reg_part}
        V_{\partial D}(a, b) \coloneqq \cb{x \in D \; \big| \; d(x, \partial D) \in (a, b]}.
    \end{align} 
    Introducing the step size parameter $h \in (0,1)$ and the threshold parameter
    \begin{equation}\label{eq:delta-order1}
    \delta(h) :=  \sqrt{8 \widehat C_b d h \log(h^{-1})},
    \end{equation}
    the "critical region" of points near the boundary is given by
    \[
        V_{\partial D}(0, \delta) \coloneqq \cb{x \in D \; \big| \; d(x, \partial D) \le \delta},
    \]
    and the adaptive time-stepping by
    \begin{equation}\label{eq:adaptive-order-1}
      \Delta t_n = \Delta t(\barX(t_n)) \coloneqq
      \begin{cases}
          h & \text{if} \quad d(\barX(t), \partial D) > \delta, \\
          h^2 & \text{otherwise.}  
      \end{cases}
  \end{equation}
  This means that a large timestep is used when the $\barX(t_n)$ is
  in the non-critical region $D\setminus V_{\partial D}(0, \delta)$
  and a small step size $h^{2}$ in the critical region near the
  boundary. (The step size used in $D^C$, whether $h$
    or $h^2$, is not of any practical importance for the output of the
    numerical method. But in our theoretical analysis we need to
    compute the numerical solution up to time $T$, and in case
    $\nu < T$, the step size for $\barX(t_n) \in D^C$ needs to be
    described to compute the solution for times in $(\nu,T]$.) The
    value of the parameter $\delta$ is chosen as a compromise between
    accuracy and computational cost: it should be sufficiently large
    so that the strong-error convergence rate in $h$ is kept at almost
    order 1, cf.~Theorem~\ref{thm:error-order1}, but it should also be
    as small as possible to keep the computational cost of the method
    low. From the proofs of Lemma~\ref{lem:strides-order1} and
    Theorem~\ref{thm:cost-order1}, it follows that the
    formula~\eqref{eq:delta-order1} is a suitable compromise for
    $\delta$.

    We are now ready to present the main results on the strong convergence rate and computational cost of the order 1 method.
    \begin{theorem}[Strong convergence rate for the order 1 method]\label{thm:error-order1}
    If Assumptions~\ref{assm:domain} and~\ref{assm:coeffs} hold, then for any for any $\xi >0$ there exists a constant $C_{\nu} >0$ such that 
    \[
    \E{ |\nu - \tau| }  \le C_{\nu} h^{1-\xi}.  
    \]
    holds for a sufficiently small $h>0$. 
    \end{theorem}
    We defer the proof to Section~\ref{subsec:proofs-order1}.

    To bound the computational cost, we first
    define the cost of a numerical solution in terms of the number time-steps used: 
    \begin{align*}
        \cost{\barX} \coloneqq \int_{0}^{\nu} \frac{1}{\Dt(\barX(t))}\rdt
    \end{align*}
    Recall that the exit time $\nu$ is different for each realization
    of the numerical solution. We make the implicit assumption that
    every evaluation of the distance of the numerical solution
    $\barX(t)$ to the boundary of the domain $\partial D$, required
    for adaptive refinement of time-step size, costs $\cO(1)$.
    
    \begin{theorem}[Computational cost for the order 1 method]\label{thm:cost-order1}
      Let Assumptions~\ref{assm:domain} and~\ref{assm:coeffs} hold and assume that one evaluation of $\Psi_1$ costs $\cO(1)$. 
      Then it holds that
      \[
        \E{\cost{\barX}} = \cO( h^{-1} \log(h^{-1}) ). 
      \]
    \end{theorem}
    We defer the proof to Section~\ref{subsec:proofs-order1}.
    
    In the special case where the It\^o-diffusion process is a standard one-dimensional Wiener process, the upper bound on the computational cost can be proven by a more fundamental approach that we outline in Appendix~\ref{appendix:a}.

    \begin{rem}
      For an SDE with low-regularity drift functions, the occupation-time on discontinuity sets is used to bound the
      computational cost of the adaptive method
      in~\cite{yaroslavtseva2022adaptive,neuenkirch2019adaptive}. When
      $d=1$, the approach~\cite{yaroslavtseva2022adaptive} carries
      over to our order 1 method, but it is an open question whether the
      approach~\cite{yaroslavtseva2022adaptive} extends to settings
      with $d>1$ for the order 1 method, and to the order 1.5 method
      in general. 
    \end{rem}
    
    \subsection{The order 1.5 method}
    \label{subsec:order_15_method}
    The general strong It\^o--Taylor scheme of order $\gamma=1.5$ is a complicated 
    expression that can be found 
    in~\cite[equation (10.4.6)]{kloeden1992stochastic}. When the so-called 
    \emph{second commutativity condition} holds: 
    \begin{equation}\label{eq:second-commutativity-condition}
        \mathcal{L}^{j_{1}} \mathcal{L}^{j_{2}} b_{i j_{3}} = \mathcal{L}^{j_{2}} \mathcal{L}^{j_{1}} b_{i j_{3}} \qquad \forall\Big( j_{1}, j_{2}, j_{3} \in \{1, \ldots, m\} \text{ and } i \in \{1, \ldots, d\}\Big),
    \end{equation}
    where the differential operator $\cL^j$ is defined in~\eqref{eq:Lj-operator},
    then the scheme simplifies to a practically useful form 
    where the computational cost of one iteration of $\Psi_{1.5}$ is $\cO(1)$, cf.~\cite[equation (10.4.15)]{kloeden1992stochastic}. 
    And in the special case of~\eqref{eq:second-commutativity-condition} when the
    diffusion coefficient is a diagonal matrix, the $i$-th component of the 
    scheme takes the form
    \begin{equation}\label{eq:order15-scheme}
        \begin{split}
            \Big(\Psi_{1.5}(\barX(t_{n}), \Delta t_n ) \Big)_i  \coloneqq &
            \barX_i(t_n) + a_i(\barX(t_n)) \Dt_n + b_{ii}(\barX(t_n)) \Delta W^i_n\\
            &+ \frac{1}{2} \cL^0 a_i(\barX(t_n)) \Dt_n^2 
            + \frac{1}{2} \cL^i b_{ii}(\barX(t_n)) \big( (\Delta W^i_n)^2 - \Dt_n \big)\\
            & + \cL^0 b_{ii}(\barX(t_n)) \big(\Delta W^i_n \Dt_n - \Delta Z^i_n\big)
            + \cL^i a_i(\barX(t_n)) \Delta Z_n^i \\
            &+ \frac{1}{2} \cL^i \cL^i b_{ii}(\barX(t_n))\Big(  \frac{(\Delta W_n^i)^2}{3} - \Delta t_n\Big)\Delta W_n^i
        \end{split}
    \end{equation}
    where we have introduced the differential operator
    \[
    \cL^0  := \sum_{i=1}^{d} a_{i} \partial_{x_{i}} + \frac{1}{2} \sum_{i,j=1}^{d} \sum_{k=1}^{m} b_{ik} b_{jk} \partial_{x_i x_j} 
    \]
    and
    \[
    \Delta Z^i_n := \int_{t_n}^{t_{n+1}} \int_{t_n}^{s_2}\rdW^i(s_1) \, \mathrm{d}s_2.  
    \]
    For computer implementations, let us add that the tuple of correlated random variables $(\Delta W^i_n, \Delta Z_n^i)$ can be generated by
    \[
    \Delta W^i_{n} = U_1 \sqrt{\Dt_n} \quad \text{and} \quad \Delta Z_n^i = \frac{1}{2} \Dt_n^{3/2} \Big( U_1 + \frac{1}{\sqrt{3}} U_2\Big), 
    \]
    where $U_1$ and $U_2$ are independent $N(0,1)$-distributed random variables.

    The step size $\Delta t_n$ for the order 1.5 method is determined adaptively by the state of the numerical solution, but with one more resolution than for the oder 1 method: A tiny step size is employed when $\barX(t_n)$ is very close to the boundary $\partial D$, a small step size is employed when $\barX(t_n)$ is slightly farther away from the boundary, and the largest step size is employed when it is far away from the boundary.
    
    To describe the adaptive time-stepping, we first introduce the step size parameter $h \in (0,1)$, the threshold parameters
    \begin{equation}\label{eq:delta-order1.5}
    \delta_1 :=  \sqrt{12 \widehat C_b d h \log(h^{-1})}, \quad \text{and} \quad 
    \delta_2 :=  \sqrt{16 \widehat C_b d h^2 \log(h^{-1})},
    \end{equation}
    and the critical regions
    \[
    V_{\partial D}(\delta_2, \delta_1) = \{x \in D \mid d(x,\partial D) \in (\delta_2, \delta_1]\}
    \]
    and 
    \[
    V_{\partial D}(0, \delta_2) = \{x \in D \mid d(x,\partial D) \le \delta_2\}.
    \]
    The time-stepping is then given by
    \begin{equation}\label{eq:adaptive-order-1.5}
      \Delta t_n = \Delta t(\barX(t_n)) \coloneqq
      \begin{cases}
          h & \text{if} \quad d(\barX(t_n), \partial D) > \delta_1 \\
          h^2 & \text{if} \quad d(\barX(t_n), \partial D) \in(\delta_2, \delta_1]\\ 
          h^3 & \text{if} \quad d(\barX(t_n), \partial D) \le \delta_2. 
      \end{cases}
  \end{equation}
  This means that the step size $h$ is used when $\barX(t_n)$
  is in the non-critical region $D\setminus V_{\partial D}(0, \delta_1)$,
  the small step size $h^2$ is used when $\barX(t_n)$ is in the critical region farthest from the boundary, and the tiny step size $h^3$ is used in when $\barX(t_n)$ is in the critical region nearest the boundary. (The step size used in $D^C$, whether $h$, $h^2$ or $h^3$ is not of any practical importance, but is needed in the theoretical analysis to extend the numerical solution up to time $T$ when $\nu < T$, similarly as for the order 1 method.) The values of the threshold parameters $(\delta_1, \delta_2)$ are chosen as a compromise between accuracy and computational cost: The critical regions should be sufficiently large so that the strong-error convergence rate in $h$ is kept at almost order 1.5, cf.~Theorem~\ref{thm:error-order1.5}, but they should also be kept as small as possible to keep the computational cost of the method low. It follows from the proofs of Lemma~\ref{lem:strides-order1.5} and Theorem~\ref{thm:cost-order1.5} that~\eqref{eq:delta-order1.5} is a suitable compromise.

    We are now ready to present the main results on the strong convergence rate and computational cost of the order 1.5 method.

    \begin{theorem}[Strong convergence rate for the order 1.5 method]\label{thm:error-order1.5}
    If Assumptions~\ref{assm:domain} and~\ref{assm:coeffs} hold, then for any 
    for any $\xi >0$ there exists a constant $C_{\nu} >0$ such that 
    \[
    \E{ |\nu - \tau| }  \le C_{\nu} h^{3/2-\xi}
    \]
    holds for sufficiently small $h>0$.
    \end{theorem}
    We defer the proof to Section \ref{subsec:proofs-order1.5}.    
    
    \begin{theorem}\label{thm:cost-order1.5}
      Let Assumptions~\ref{assm:domain} and~\ref{assm:coeffs} hold 
      and assume that one evaluation of $\Psi_{1.5}$ costs $\cO(1)$.
      Then it holds that 
      \[
        \E{\cost{\barX}} = \cO\big( h^{-1}  \log(h^{-1}) \big). 
      \]
    \end{theorem}

\section{Theory and proofs} \label{sec:theory} This section proves
    theoretical properties of the adaptive time-stepping methods of
    order 1 and 1.5. We first describe how critical regions combined
    with adaptive time-stepping can bound the overshoot of the
    diffusion process with high probability, and thereafter use this
    property to prove the strong convergence of the method given by
    Theorems~\ref{thm:error-order1}
    and~\ref{thm:error-order1.5}. Lastly, we prove upper bounds for
    the expected computational cost of the methods.

Let us first state a few useful theoretical results. 

\begin{proposition}\label{prop:conv-rate-strong}
    Let Assumption~\ref{assm:coeffs} hold. Recall that $X$ denotes the exact solution of the SDE~\eqref{eq:SDE} and that 
    $\barX$ denotes the numerical solution computed on the adaptive mesh $\mesh{\Dt}$ with the strong It\^o--Taylor scheme of order $\gamma \in \{1, 1.5\}$. Then for any $p \ge 1$, the following bound holds:
    \[
      \E{ \sup_{t_k \in \mesh{\Dt}} |X(t_k) - \barX(t_k)|^p } \le C h^{ \gamma p}, 
    \]
    where $C>0$ depends on $p$. 
\end{proposition}
See~\cite[Theorem 10.6.3]{kloeden1992stochastic} for a proof Proposition~\ref{prop:conv-rate-strong}.

\begin{proposition}\label{prop:errorInterpolation}
For the exact solution of the SDE~\eqref{eq:SDE} and a sufficiently small $h>0$, it holds for any $p \ge 1$ that
\[
\E{ \max_{k\in \{0,1,\ldots, \lceil T/ h\rceil\}} \sup_{s \in [0, h] }  
|X(kh+s) - X(k h)|^p } \le C \sqrt{ h^p \log( h^{-1})},    
\]
where $C>0$ depends on $p$. 
\end{proposition}
The above proposition is a direct consequence of replacing the piecewise constant Euler--Maruyama approximation with a piecewise constant interpolation of the exact solution of the SDE in~\cite[Theorem 2]{Muller-Gronbach2002}.  

\subsection{Order 1 method} \label{subsec:proofs-order1}
This section proves theoretical results for the order 1 method. 

For the It\^o process~\eqref{eq:SDE} and $t>s\ge0$, let
\begin{equation}
    \label{eq:runningmaximum}
    M(s,t) \coloneqq \sup_{r \in [s, t]} \abs{X(r) - X(s)}.
\end{equation}
One may view $M(s,t)$ as the maximum stride the process $X$ takes over the interval $[s,t]$. The following lemma below shows that the threshold parameter 
$\delta$ is chosen sufficiently large to ensure that the maximum stride $X$ takes over every interval in the mesh $\mesh{\Dt}$ is with very high probability bounded by $\delta$. This estimate will help us bound the probability that the numerical solution exits the domain $D$ from the non-critical region in the proof of Theorem~\ref{thm:error-order1} (i.e., to bound the probability of exiting $D$ when using a large timestep).

\begin{lemma}\label{lem:strides-order1}
Let Assumptions~\ref{assm:domain} and~\ref{assm:coeffs} hold, and assume that 
the timestep parameter $h \in (0,1)$ is sufficiently small.
For the threshold parameter $\delta = \sqrt{8 \widehat C_b h d\, \log(h^{-1})}$
and the maximal-stride set 
\[
A := \Big\{\omega \in \Omega \mid M(t_n, t_{n} + \Delta t_n) \le \delta \quad \forall n \in \{0, 1,\ldots, N-1\} \Big\},
\]
it then holds that $\bP(A) = 1 - \cO(h)$.
\end{lemma}

\begin{proof}
  From the adaptive time-stepping, we know that the mesh $\mesh{\Dt}$
  contains $N = |\mesh{\Dt}|-1$ many intervals where $N$ is a random
  integer that is bounded from below by $T/h$ and from above by
  $T/h^2$. At most $T/h$ of the intervals are of length $h$ and at
  most $T/h^2$ of the intervals are of length $h^2$. To avoid
  complications due to a random number of elements in the mesh, we
  extend the mesh $\mesh{\Dt}$ to span over $[0,2T]$ in such a way
  that the extended mesh agrees with $\mesh{\Dt}$ over the interval
  $[0,t_N]$ and contains exactly $T/h$ many intervals of length $h$
  and $T/h^2$ many intervals of length $h^2$. In other words,
\[
\mesh{E} = \{t_0, t_1,\ldots, t_N, t_{N+1}, \ldots , t_{\widehat N}\}
\]
where $\mesh{E} \cap \mesh{\Dt} = \mesh{\Dt}$ and $t_{\widehat N} = 2T$
with $\widehat N = T/h + T/h^2$.
And for
\[
\widehat \Delta^1 := \left\{ k \in \{0,1,\ldots, \widehat N-1\} \mid  \Delta t_k = h \right\} 
\]
and
\[
\widehat \Delta^{2} := \left\{ k \in \{0,1,\ldots, \widehat N-1\} \mid  \Delta t_k = h^2  \right\} 
\]
we have that $|\widehat \Delta^1| = T/h$ and $|\widehat \Delta^2| = T/h^2$,
respectively. 
We represent these two sets of 
integers and relabel their associated mesh points
as follows:
\[
\widehat \Delta^1 = \{ \widehat \Delta^1(1), \widehat\Delta^1(2),\ldots, \widehat\Delta^1({T/h}) \}
\quad \text{with} 
\quad  t^h_n := t_{\widehat\Delta^1(n)}\qquad n \in \{1,2,\ldots, T/h\}
\]
and 
\[
\widehat \Delta^{2} = \{\widehat\Delta^{2}(1), \widehat\Delta^{2}(2),\ldots, \widehat\Delta^{2}({T/h^2}) \}
\quad \text{with} \quad 
t^{h^2}_n := t_{\widehat\Delta^{2}(n)}   \qquad n \in \{1,2,\ldots, T/h^2\}.
\]

Introducing the maximal-stride set over the extended mesh
\begin{equation}
    B \coloneqq \cb{\omega \in \Omega \; \big| \; M(t_{n}, t_{n+1}) \le \delta \quad \forall n \in \{0,1,\ldots, \widehat N-1\} },
\end{equation}
and noting that $B \subset A$, we achieve the following bound for the probability of $A^{C}$: 
\begin{equation}\label{eq:prob-A-complement1}
    \begin{split}
        \Prob{A^{C}} &\leq \Prob{B^{C}}\\
        &\leq 
        \sum_{n=0}^{\hat N-1} \Prob{M(t_{n}, t_{n} + \Dt_{n}) > \delta}\\
        &=\sum_{n =1}^{T/h} \Prob{M(t^{h}_{n}, t^{h}_{n} + h) > \delta}
        + \sum_{n = 1}^{T/h^2} \Prob{M(t_{n}^{h^2}, t^{h^2}_{n} + h^2) > \delta}.
    \end{split}
\end{equation}
Recall that 
the integral form of the SDE~\eqref{eq:SDE} is given by
\begin{equation}
X(r) = X(s) + \int_{s}^{r} a(X(u)) du + \int_{s}^{r} b(X(u))\rdW(u).
\end{equation}
and let 
\[
C_a := \sup_{x \in \bR^d } |a(x)|.
\]
This yields 
\[
M(t^{h}_{n}, t^{h}_{n} + h) = \sup_{r \in [t_n, t_n+h]}
|X(r) - X(s)| \le C_a h + \sup_{r \in [0, h]} \left|\int_{t_n^h}^{t_n^h+r} b(X(u))\rdW(u)\right|.
\]
Assuming that $h$ is sufficiently small so that
\[
\delta - C_a h \ge \frac{\delta }{\sqrt{2}},
\]
we obtain that 
\[
\Prob{M(t^{h}_{n}, t^{h}_{n} + h) > \delta}
\le \Prob{\sup_{r \in [0, h]} \abs{\int_{t_n^h}^{t_n^h+r} b(X(u))\rdW(u)} \ge \delta/\sqrt{2}}.
\]
Introducing $\tilde{b}(u):= b(X(u+t_n))$ and $\tildeW(t):= W(t+t_n^h) - W(t_n^h)$, the above integral takes the form
\[
\int_{t_n^h}^{t_n^h+r} b(X(u))\rdW(u) = \int_0^r \tilde{b}(u) \, d\tildeW(u).
\]
Since $t_n^h$ is an element in the mesh $\mesh{E}$, it is a finite stopping time, and the strong Markov property therefore implies that $\tildeW(t)$ is a standard Wiener process associated to the filtration $\sigma(\{\tildeW(t)\}_{u \in[0,r]} ) \subset \cF_{t_n+r}$, cf.~\cite[Theorem 2.16]{morters2010brownian}. Furthermore, the integrand $\tilde{b}(u)$ is a square-integrable and $\cF_{t_n+u}$-adapted stochastic process, and Assumption~\ref{assm:coeffs}.3 implies that
\[
|\xi|^{-2}\int_0^h \xi^\top \tilde{b}(u) \tilde{b}^\top(u) \xi \, du \le \widehat C_b h 
\quad \quad \forall \xi \in \bR^d\setminus\{0\}.
\]
By~\cite[Proposition 8.7]{baldi2017stochastic},
Doob's martingale inequality then yields that
\[
\Prob{\sup_{r \in [0, h]} \abs{\int_{0}^{r} \tilde b(u) d\tildeW(u)} \ge \delta/\sqrt{2}} \le 2 d \exp\left(-\frac{\delta^2}{4 \widehat C_b h d} \right)
= 2 d h^2
\]
for any $n \in \{1, 2,\ldots, T/h\}$. Assuming $h\le 2/3$, a similar argument yields that
\[
\Prob{M(t_{n}^{h^2}, t^{h^2}_{n} + h^2) > \delta}
\le 2 d h^3 \quad \forall n \in \{1,2,\ldots, T/h^2\}.
\]
Inequality~\eqref{eq:prob-A-complement1} yields that
\[
\bP(A^C) \le \sum_{n = 1}^{T/h} \Prob{M(t_{n}^{h}, t^{h}_{n} + h) > \delta} + \sum_{n = 1}^{T/h^2} \Prob{M(t_{n}^{h^2}, t^{h^2}_{n} + h^2) > \delta} \le 4dT h.  
\]

\end{proof}

\begin{proof}[Proof of Theorem~\ref{thm:error-order1}]
We first partition the exit-time error into two parts:
\[
\E{|\tau - \nu|} = \E{ |\tau - \nu| 1_{\nu< \tau}} + \E{ |\tau - \nu| 1_{\nu>\tau}} =: I + II.     
\]

Since $\nu \in \mesh{\Dt}$, Proposition~\ref{prop:conv-rate-strong} for $\gamma =1$ implies that  
\[
    \E{|X(\nu) - \barX(\nu)|} \le
    \E{\max_{t_k \in \mesh{\Dt}} |X(t_k) - \barX(t_k)|} = \cO(h).     
\] 
For term $I$, $\nu < \tau$ implies that $\nu <T$. Consequently, $\barX(\nu) \in D^C$
and $X(\nu) \in D$, so there exists a $y \in \partial D$ satisfying that 
$|X(\nu) - y| \le |X(\nu) -\barX(\nu)|$, and of course also that $\tau^{\nu, y} = 0$. Thanks to the Lipschitz property~\eqref{eq:lipschitz-u}, we obtain that  
\[
    \begin{split}
  I  &=
  \E{ \E{ (\tau - \nu) 1_{\nu< \tau} \mid \cF_\nu} }\\ 
  &= \E{ \E{ \tau^{\nu, X(\nu)} 1_{\nu< \tau} \mid \cF_\nu} }\\
  &\le \E{ \E{ \tau^{\nu, X(\nu)}  - \tau^{\nu, y }\mid \cF_\nu} } \\
  &=  \E{ u(\nu, X(\nu)) - u(\nu, y)} \\
  & \le L \E{ |X(\nu) - y| }\\
  & \le L \E{ |X(\nu) - \barX(\nu)| }\\
& = \cO(h).
    \end{split} 
\]

For the second term, we assume for the given $\xi>0$ that $r:=h^{1-\xi} < \bar R_D$, where 
we recall that $\bar R_D$ is defined in Assumption~\ref{assm:coeffs}.2, and 
introduce the second exit time problem 
\[
\tau_r = \inf\{t\ge 0 \mid X(t) \notin D_r\} \wedge T.
\]
As $r<\bar R_D$, Proposition~\ref{prop:fk2} applies, 
which in particular means that 
the function $u_r(t,x) = \E{\tau_{r}^{t,x} }$ satisfies the Lipschitz property~\eqref{eq:lipschitz-u}.
 
Noting that $\tau_r \ge \tau$, we obtain 
\[
  II = \E{ (\nu - \tau) \ind{\nu > \tau} } \le 
  \E{ \tau_r -  \tau } + \E{ (\nu - \tau_r) \ind{\nu > \tau}}   =: II_1 + II_2. 
\]
Here,
\[
  II_1 = \E{ \E{\tau_r -  \tau \mid \cF_\tau }  } = \E{ \E{\tau_r^{\tau,X(\tau)} \mid \cF_\tau }  } = \E{ u_r(\tau, X(\tau)) }
  \le L r = \cO(h^{1-\xi}), 
\]
where the last inequality follows from~\eqref{eq:lipschitz-u}
and
\begin{equation}\label{eq:distanceStatement}
  X(\tau) \in\partial D \implies d(X(\tau), \partial D_r) = r.
\end{equation}
The statement~\eqref{eq:distanceStatement} is due to the
diffeomorphism~\eqref{eq:diffeomorphism}, as it tells us that 
whenever $\tau< T$ and thus $X(\tau) \in \partial D$, we may 
view the diffeomorphism as a projection onto the boundary of $D_r$:
\begin{equation}\label{eq:projection-pi}
  \partial D \ni X(\tau) \mapsto \pi_{r}(X(\tau)) \in \partial D_r \qquad \text{satisfying that} \quad 
  |\pi_{r}(X(\tau)) - X(\tau) | = r.
\end{equation}
Using that $u_r(T, \cdot) = 0$ and 
$u_r(\tau, \pi_{r}(X(\tau))) = 0$ whenever $\tau <T$, we verify the last inequality 
for $II_1$ as follows:
\[
  \begin{split}
\E{ u_r(\tau, X(\tau))} &= \E{ u_r(\tau, X(\tau)) \ind{\tau <T}}\\
& = \E{ \big|\, u_r(\tau, X(\tau)) - u_r(\tau, \pi_{r}(X(\tau))) \,\big| \ind{\tau <T} }\\
& \le \E{ L  \big|\, X(\tau) -\pi_{r}(X(\tau)) \,\big| \ind{\tau <T} } \\
&\le L r.
  \end{split}
\]

To bound $II_2$, we first note that since 
$(\nu - \tau_r) \ind{\tau_r = T} \le 0$, 
it holds that 
\[
II_2 \le \E{ (\nu - \tau_r) \ind{ \{\nu > \tau\} \cap \{ \tau_r < T\}}}.  
\]
Let us introduce
\[
  t^* := \max\{t_k \in \mesh{\Dt} \mid t_k \le \tau_r \}
\]
and let $A$ denote the maximal-stride set defined in Lemma~\ref{lem:strides-order1}, for which we recall that $\bP(A) = 1 - \cO(h)$. 
Since $\tau_r - t^* \le h$, it holds that 
\[
| X(\tau_r, \omega)- X(t^*, \omega)| \le \delta \qquad \forall \omega \in A
\]
and since $\tau_r <T \implies X(\tau_r) \in \partial D_r$, we also have that
\begin{equation*}\label{eq:x-star-dist-bound1}
d(X(t^*, \omega), \partial D_r) \le \delta \qquad \omega \in A \cap \{\tau_r <T\}.
\end{equation*}

To bound the distance between the exact process and the numerical one
at time $t^*$, let $p^* \in \bN$ be sufficiently large so that $p^*\xi > 1$. Then by Proposition~\ref{prop:conv-rate-strong}
\begin{equation}\label{eq:dist-numerical-path-true-path-order1}
  \bP(|\barX(t^*) - X(t^*)| \ge r)
  \le \frac{\E{|\barX(t^*) - X(t^*)|^{p^*}}}{h^{(1-\xi) p^*}} = \cO(h).
\end{equation}
Let further $\widetilde{A} := \{\omega \in A \mid |\barX(t^*) - X(t^*)| < r\}$, and note that $\bP(\widetilde{A}) = 1 - \cO(h)$. 
We will next show that for all paths in $\widetilde A \cap \{\tau_r <T\}$, the numerical solution uses 
the smallest timestep at $t^*$, meaning that 
\begin{equation}\label{eq:timestep-B}
\Delta t\big(\barX(t^*,\omega)\big) = h^2 \quad  \forall \omega \in \widetilde{A} \cap \{\tau_r <T\}.
\end{equation}
Recall first that the non-critical region of $D$ for the order 1 method is given by $\widetilde D = D \setminus V_{\partial D}(0,\delta)$,
and observe that for all paths $\omega \in \widetilde{A} \cap \{\tau_r <T\}$, it holds that  
\[
d(\barX(t^*, \omega), \partial D_r) \le d(\barX(t^*, \omega), X(t^*)) + d(X(t^*),\partial D_r) < \delta +r.
\]
Since $d(\widetilde D, \partial D_r) \ge \delta +r$, we conclude that
$ \barX(t^*, \omega) \notin \widetilde D$ and~\eqref{eq:timestep-B} is verified.
Thanks to~\eqref{eq:timestep-B}, we can sharply estimate the distance 
between $\barX(t^*)$ and $X(\tau_r)$ as follows:
\[
  \begin{split}
  \bP\Big( \{|&\barX(t^*) - X(\tau_r)| \ge r \} \cap \{\tau_r <T\}\Big) \\ &\le
  \bP\Big( \{|\barX(t^*) - X(\tau_r)| \ge h^{1-\xi} \} \cap \{\tau_r <T\} \cap \widetilde{A} \Big) + \cO(h)\\
  &\le 
  \frac{\E{|\barX(t^*) - X(\tau_r)|^{p^*} \ind{\widetilde{A} \cap \{\tau_r <T\}}}}{ h^{(1-\xi)p^*} } + \cO(h)\\
  &\le \frac{\E{ p^*|\barX(t^*) - X(t^*)|^{p^*} \ind{\widetilde{A} \cap \{\tau_r <T\}} + p^*|X(t^*) - X(\tau_r)|^{p^*}\ind{\widetilde{A} \cap \{\tau_r <T\}} }}{ h^{(1-\xi)p^*} }
  +\cO(h) \\
  &= \cO(h).
  \end{split}
\]
The first summand in the last inequality is bounded by~\eqref{eq:dist-numerical-path-true-path-order1} 
and the second one is bounded by Proposition~\ref{prop:errorInterpolation}, equation~\eqref{eq:timestep-B}
(which implies that $\tau_r-t^* \le h^2$ for all $\omega \in \widetilde A \cap \{\tau_r <T\}$), and
\[
  \begin{split}
  \E{|X(t^*) - X(\tau_r)|^{p^*}\ind{ \widetilde{A}\cap \{\tau_r <T\}}} &\le
  \E{ \max_{k\in \{0,\ldots, T/h^2 -1\}} \sup_{s \in [0,h^2] }
    |X(kh^2 + s) - X(k h^2)|^{p^*} }\\
  &= \cO( h^{p^*} \sqrt{\log(h^{-1})} )\,.
  \end{split}
\]
For $G := \{\omega \in \Omega \mid |\barX(t^*) - X(\tau_r)| < r \}$, we conclude that 
$\bP(G^C \cap \{\tau <T\}) = \cO(h)$ and  
\[
  \begin{split}
     \omega \in G \cap \{\tau_r < T\} &\implies d(\barX(t^*,\omega), \partial D_r) < r\\
    &\implies \barX(t^*,\omega) \notin D\\
    &\implies \nu(\omega) \le t^*(\omega)\\
    &\implies \nu(\omega) \le \tau_r(\omega).
  \end{split}
\]
Consequently, 
\[
\begin{split}
  II_2 &\le  \E{ (\nu - \tau_r) \ind{\{\nu > \tau\} \cap \{\tau_r < T\}}}\\
  &\le  \underbrace{\E{ (\nu - \tau_r) \ind{\{\nu > \tau\} \cap \{\tau_r < T\}\cap G}}}_{\le 0} +  T  \, \bP(G^C \cap \{\tau_r <T\})\\
  & = \cO(h).
 \end{split}
\]
\end{proof} 

We next prove the computational cost result for the order 1 method.

\begin{proof}[Proof of Theorem~\ref{thm:cost-order1}]

  Let $s_n := n h^2$ for $n=0,1,\ldots$ denote a set of deterministic uniformly spaced mesh points. This mesh contains all realizations of the adaptive mesh, meaning that $\mesh{\Dt}(\omega)\subset \{s_n\}_{n\ge 0}$ for all $\omega \in \Omega$,
  and we have that  
\begin{equation}\label{eq:costOrder1}
  \begin{split}
\E{\text{Cost}(\barX)} &= \E{\int_{0}^{\nu} \frac{1}{\Dt(\barX(t))}\rdt}\\
& \le  \int_0^T \frac{\E{\ind{\{\nu}>t\} \cap \{ \Dt(\barX(t)) = h \}}}{h}
+ \frac{\E{\ind{\{\tau_{2\delta}>t\} \cap \{ \Dt(\barX(t)) = h^2 \}}}}{h^2} \,\rdt\\
&\le  \frac{T}{h} +  \sum_{n=0}^{h^{-2}-1} \bP\big(\{\nu>s_{n}\} \cap \{ \Dt(\barX(s_n)) = h^2 \}\big).
\end{split}
\end{equation}
To bound the second term we assume the step size parameter $h>0$ is sufficiently small such that $\delta < \bar R_D/2$, where we recall that $\bar R_D$ is defined in Assumption~\ref{assm:coeffs}, and consider the stopping time
\[
  \tau_{2\delta} = \inf\{ t \ge 0 \mid X(t) \notin D_{2 \delta} \} \wedge T. 
\] 
Let further $A$ denote the maximal-stride set defined in Lemma~\ref{lem:strides-order1} and let $B:= \{ \omega \mid \max_{t_k \in \mesh{\Dt}} |X(t_k) - \barX(t_k)| \le \delta \}$. Then it holds that
\[
\omega \in A\cap B \implies \nu(\omega) \le \tau_{2\delta}(\omega),  
\]
which we prove by contradiction as follows: Suppose that $\omega \in A \cap B$ and $\tau_{2\delta}<\nu$. Then $\barX(\tau_{2\delta}) \in D$ and $X(\tau_{2\delta}) \in \partial D_{2\delta}$. Let $t_k$ denote largest mesh point in $\mesh{\Dt}$ that is smaller or equal to $\tau_{2\delta}$. Then $(t_k,\tau_{2\delta}) \subset(t_k,t_{k+1})$ and $\omega \in A$ implies that $X(t_k) \notin D_{\delta}$ while we have that $\barX(t_k) = \barX(\tau_{2\delta}) \in D$.
Hence $d(\barX(t_k), X(t_k)) > \delta$ which contradicts that $\omega \in B$. 

From Proposition~\ref{prop:conv-rate-strong} we have that 
\begin{equation*}
  \begin{split}
\Prob{B^C} &\le \E{ \frac{\max_{t_k \in \mesh{\Dt}} |X(t_k) - \barX(t_k)|^2}{\delta^2} }\\
& = \cO\left( \frac{h}{\log(h^{-1})} \right) = o(h),
  \end{split}
\end{equation*}
and since $\bP(A) = 1- \cO(h)$, we conclude that $\Prob{A\cap B} \le 1 - \cO(h)$. 
Recalling further that 
\[
  \Dt(\barX(s_n)) = h^2  \iff d(\barX(s_n), \partial D) \le \delta, 
\]
and
\[
  B \cap \{ d(\barX(s_n), \partial D) \le \delta \} 
  \subset   B \cap \{ d(X(s_n), \partial D) \le 2\delta \},
\]
we obtain that
\begin{equation}\label{eq:sumSplitCost}
  \begin{split} 
  \bP\big(\{\nu>s_{n}\} \cap \{ \Dt(\barX(s_n)) &= h^2 \}\big)\\
  & = \bP\big(A \cap B \cap \{\nu>s_{n}\} \cap \{ d(\barX(s_n), \partial D) \le \delta \}\big) + \cO(h) \\
  & \le  \bP\big(\{\tau_{2\delta}>s_{n}\} \cap \{ d(X(s_n), \partial D) \le 2\delta \}\big) + \cO(h).
  \end{split}
\end{equation}
To bound the first summand, note first that the "density" of the SDE~\eqref{eq:SDE} on the domain $D_{2\delta}$ with paths removed when they exit $D_{2\delta}$ is a generalized function 
$p(t,x): [0,T]\times \overline D_{2\delta} \to [0,\infty]$ that solves the following 
absorbing-boundary Fokker--Planck equation~\cite{naeh90,schuss80}: 
\begin{equation}\label{eq:fokkerPlanckPDE}
\left.\begin{aligned}
    \partial_t p &=  -\nabla \cdot (a p) + \frac{1}{2}  \nabla \cdot \rb{bb^{\top} \nabla p}    && \text{in}  \quad (0,T] \times D_{2\delta}, \\
   p  &= 0  && \text{on} \quad  (0,T] \times \partial D_{2\delta}\\
   p(0,x) &= \delta(x-x_0) && \quad  x\in \overline D_{2\delta}. 
\end{aligned}\right\}
\end{equation}
By the regularity constraints imposed on coefficients 
in Assumption~\ref{assm:coeffs} with $2\delta \le \bar R_D$
and since the boundary $\partial D_{2\delta}$ is $C^3$,  
the Fokker--Planck equation has a unique 
solution that blows up at $(t,x) = (0,x_0)$, and the solution 
may be viewed as the Green's function $p(t,x) = G(t,x;0,x_0)$,~cf.~\cite[Chapter 3.7]{friedman64}. Note that the well-posedness holds for any $\delta \in [0,\bar R_D/2)$, and that we are considering a parametrized set of solutions $p(x,t;\delta)$
to the PDE~\eqref{eq:fokkerPlanckPDE}, where we supress the dependence on $\delta$
when confusion is not possible. 

Considering $p$ as a mapping from a smaller domain where a neighborhood 
$N$ of the singular point $(0,x_0)$ is removed from the full domain, it holds that $p:([0,T] \times \overline{D}_{2\delta}) \setminus N \to [0,\infty)$ is a continuously differentiable function, cf.~\cite[Chapters 1 and 3.7]{friedman64}.
Since we are interested in the propeties of $p$ near the cylindrical boundary 
$[0,T]\times \partial D_{2\delta}$, we will proceed as follows to remove a cylindrical neighborhood containing $(0, x_0)$ from the domain of $p$. 
For $\lambda := d(x_0, \partial D)/2$, let $\bar h>0$ be an upper bound for the step size parameter, such that 
$d(x_0, \partial D_{-2\delta}) > \lambda$ and $\delta(h) < \bar R_D/2$ hold 
whenever $h \le \bar h$. For $\Gamma_{2\delta} := D_{2\delta} \setminus D_{-2\delta}$
it then holds that $d(x_0, \Gamma_{2\delta}) >\lambda$ and Lemma~\ref{lem:gradPBound} implies that there exists a constant $C_p>0$ that is uniform in $h \in [0,\bar h]$ such that 
\begin{equation} \label{eq:pBoundOrder1}
  \max_{(t,x) \in [0,T] \times \overline{\Gamma}_{2\delta}} p(t,x) \le 4C_p \delta.     
\end{equation}
And Lemma~\ref{lem:cI_bounded} implies there exists a constant $C_{\cI}>0$ such that 
\[
\max_{r \in [-\bar R_D, \bar R_D]} \int_{\partial D_{r}} dS(x)  \le C_{\cI} . 
\]
%
%
Thanks to the co-area forumla,  
\[
  \begin{split}
  \bP(\{\tau_{2\delta}>s_{n}\} \cap \{ d(X(s_n), \partial D) \le 2 \delta \}) 
  &= \int_{\Gamma_{2\delta}}  p(s_n ,x) dx\\
  &= \int_{-2\delta}^{2\delta}  \int_{\partial D_{r}} p(s_n,x) dS(x) dr\\ 
  &\stackrel{\eqref{eq:pBoundOrder1}}{\le} 4C_p  \int_{-2\delta}^{2 \delta}  \delta \int_{\partial D_{r}}  dS(x) dr \\
  & \le 16 C_p C_\cI \, \delta^2,
  \end{split}
\]
and it follows from~\eqref{eq:costOrder1} and~\eqref{eq:sumSplitCost} that 
\[
  \begin{split}
\E{\cost{\barX}} &\le  \cO(h^{-1}) + \sum_{n=0}^{h^{-2}-1} \bP(\{\tau_{2\delta}>s_{n}\} \cap \{d(X(s_n), \partial D) \le 2\delta) \\ 
&= \cO\big(h^{-2} \delta^2 \big) = \cO(h^{-1} \log(h^{-1}))\, .
  \end{split}
\]
\end{proof}

\subsection{Order 1.5 method} \label{subsec:proofs-order1.5}
This section proves theoretical results for the order 1.5 method.

The following lemma below shows that the threshold parameters 
$\delta_1$ and $\delta_2$ are chosen sufficiently large to ensure 
that the maximum stride $X$ takes over every interval in the mesh $\mesh{\Dt}$ is with high likelihood bounded by either $\delta_1$ or $\delta_2$, depending on the length of the interval. The estimate will 
help us bound the probability that the numerical solution exits the domain 
$D$ from anywhere but the critical region nearest the boundary of $D$ 
in the proof of Theorem~\ref{thm:error-order1.5} (i.e., bound the probability of exiting $D$ using a larger timestep than $h^3$).

\begin{lemma}\label{lem:strides-order1.5}
Let Assumptions~\ref{assm:domain} and~\ref{assm:coeffs} hold and assume that 
the parameter $h \in (0,1)$ is sufficiently small.
For the time-steps in the mesh $\mesh{\Dt}$, let
\[
\begin{split}
\Delta^j &:= \left\{ k \in \{0,1,\ldots, N-1\} \mid  \Delta t_k = h^j \right\} 
\quad \text{for} \quad j=1,2,3.
\end{split}
\]
For the threshold parameters 
\[
\delta_1 = \sqrt{12 \widehat C_b d h \, \log(h^{-1})} \quad 
\text{and} \quad 
\delta_2 = \sqrt{16 \widehat C_b d h^2\, \log(h^{-1})}
\]

and the maximal-stride sets 
\[
\begin{split}
A_1 &:= \Big\{\omega \in \Omega \mid M(t_n, t_{n} + \Delta t_n) \le \delta_1 \quad \forall n \in \Delta^1  \Big\} \\
A_2 &:= \Big\{\omega \in \Omega \mid M(t_n, t_{n} + \Delta t_n) \le \delta_2 \quad \forall n \in \Delta^{2} \cup \Delta^{3}  \Big\}
\end{split}
\]
it then holds for $A := A_1\cap A_2$ that $\bP(A) = 1 - \cO(h^{2})$.
\end{lemma}

\begin{proof}
From the adaptive time-stepping, we know that the mesh $\mesh{\Dt}$ contains $N = |\mesh{\Dt}|-1$ many intervals where $N$ is a random integer that is bounded from below by $T/h$ and from above by $T/h^3$. At most $T/h$ of the time-steps are of length $h$, at most $T/h^2$ are of length $h^2$, and at most $T/h^3$ are of length $h^3$. To avoid complications due to a random number of elements in the mesh, we extend the mesh $\mesh{\Dt}$ to span over $[0,3T]$ in such a way that the extended mesh agrees with $\mesh{\Dt}$ over the interval $[0,t_N]$ and contains exactly $T/h$ many time-steps of length $h$,  $T/h^2$ many time-steps of length $h^2$ and $T/h^3$ many time-steps of length $h^3$. In other words,
\[
\mesh{E} = \{t_0, t_1,\ldots, t_N, t_{N+1}, \ldots , t_{\widehat N}\}
\]
where $\mesh{E} \cap \mesh{\Dt} = \mesh{\Dt}$ and $t_{\widehat N} = 3T$
with $\widehat N = T/h + T/h^2 + T/h^3$. 
And for
\[
\widehat \Delta^j := \left\{ k \in \{0,1,\ldots, \widehat N-1\} \mid  \Delta t_k = h^j \right\} \quad 
\text{for} \quad  j =1,2,3,
\]
it holds that 
\[
|\widehat \Delta^j| = \frac{T}{h^j} \quad \text{for} \quad j=1,2,3.
\]
We represent these three sets of integers and relabel their associated mesh points as follows: For $j=1,2,3$, let
\[
\widehat \Delta^j = \{ \widehat \Delta^j(1), \widehat\Delta^j(2),\ldots, \widehat\Delta^j({T/h^{j}}) \}
\quad \text{with} 
\quad  t^{h^j}_n := t_{\widehat\Delta^j(n)}\qquad n \in \{1,2,\ldots, T/h^j\}.
\]

Introducing the following maximal-stride sets over the 
extended mesh
\begin{equation}
\begin{split}
    B_1 &:= \cb{\omega \in \Omega \; \big| \; M(t_{n}, t_{n+1}) \le \delta_1 \quad \forall n \in \widehat \Delta^1 }\\
    B_2 &:= \cb{\omega \in \Omega \; \big| \; M(t_{n}, t_{n+1}) \le \delta_2 \quad \forall n \in \widehat \Delta^2 \cup \widehat \Delta^3 }
    \end{split}
\end{equation}
and noting that $B := B_1 \cap B_2$ is a subset of $A$, we obtain 
the following upper bound: 
\begin{equation}\label{eq:prob-A-complement}
    \begin{split}
        \Prob{A^{C}} &\leq \Prob{B^{C}}\\
        &\leq 
        \sum_{n =1}^{T/h} \Prob{M(t^{h}_{n}, t^{h}_{n} + h) > \delta_1}
        + \sum_{n = 1}^{T/h^2} \Prob{M(t_{n}^{h^2}, t^{h^2}_{n} + h^2) > \delta_2}\\
        & \quad + \sum_{n = 1}^{T/h^3} \Prob{M(t_{n}^{h^3}, t^{h^3}_{n} + h^3) > \delta_2}.
    \end{split}
\end{equation}
Assuming that $h \in (0,1)$ is sufficiently small so that  
\[
\delta_1 - C_a h \ge \frac{\delta_1}{\sqrt{2}}, \qquad \delta_2 - C_a h^2 \ge \frac{\delta_2}{\sqrt{2}} 
\qquad \text{and} \quad h \le \frac{4}{5},
\]
then a similar use of Doob's martingale inequality as in the proof of Lemma~\ref{lem:strides-order1} yields  
\begin{align*}
\Prob{M(t^{h}_{n}, t^{h}_{n} + h) > \delta_1} &\le 2 d h^{3}  && \forall n \in \{1, 2,\ldots, T/h\},\\
\Prob{M(t^{h^2}_{n}, t^{h^2}_{n} + h^2) > \delta_2} 
&\le 2d \exp\left(-\frac{\delta_2^2}{4 \widehat C_b h^2 d} \right) 
= 2 d h^{4} 
&& \forall n \in \{1, 2,\ldots, T/h^2\},\\
\Prob{M(t^{h^3}_{n}, t^{h^3}_{n} + h^3) > \delta_2} 
&\le 2d \exp\left(-\frac{\delta_2^2}{4 \widehat C_b h^3 d} \right) 
\le 2 d h^{5} 
&&\forall n \in \{1, 2,\ldots, T/h^3\}.
\end{align*}
Conclusion: $\bP(A^C) \le 6 d\, T \,h^{2}$.   

\end{proof}

\begin{proof}[Abbreviated proof of Theorem~\ref{thm:error-order1.5}]

The exit-time error can be partitioned into two parts:
\[
\E{|\tau - \nu|} = \E{ |\tau - \nu| 1_{\nu< \tau}} 
+ \E{ |\tau - \nu| 1_{\nu>\tau}} =: I + II.    
\]
Similarly as in the proof of Theorem~\ref{thm:error-order1}, but now using the 
strong It\^o--Taylor method of order $\gamma =1.5$, we obtain that 
\[
  I  \le L \E{ |X(\nu) - \barX(\nu)| }  = \cO(h^{3/2}).
\]

For the second term, assume for the given $\xi>0$ that $r:=h^{3/2-\xi} < \bar R_D$ 
and introduce the second exit time problem 
\[
\tau_r = \inf\{t\ge 0 \mid X(t) \notin D_r\} \wedge T.
\]
As Proposition~\ref{prop:fk2} applies, we recall that
$u_r(t,x) = \E{\tau_{r}^{t,x} }$ satisfies the Lipschitz property~\eqref{eq:lipschitz-u}.
 Noting that $\tau_r \ge \tau$, we obtain 
\[
  II = \E{ (\nu - \tau) \ind{\nu > \tau} } \le 
  \E{ \tau_r -  \tau } + \E{ (\nu - \tau_r) \ind{\nu > \tau}}   =: II_1 + II_2. 
\]
A similar argument as in the proof of Theorem~\ref{thm:error-order1} yields that
\[
  II_1 = \E{ \E{\tau_r -  \tau \mid \cF_\tau }  } = \E{ u_r(\tau, X(\tau)) }
  \le L r = \cO(h^{3/2-\xi}). 
\]

To bound $II_2$, first observe that $(\nu - \tau_r) \ind{\tau_r = T} \le 0$
implies that
\[
II_2 \le \E{ (\nu - \tau_r) \ind{ \{\nu > \tau\} \cap \{ \tau_r < T\}}}.  
\]
We introduce 
\begin{equation}\label{eq:t-star-order1.5}
  t^* := \max\{t_k \in \mesh{\Dt} \mid t_k \le \tau_r \}
\end{equation}
and note for later reference that $\tau_r - t^* \le \Delta t(\barX(t^*))$.
Let $A$ denote the maximal-stride set defined in Lemma~\ref{lem:strides-order1.5}, 
where we recall that $\bP(A) = 1 - \cO(h^{2})$. Since $\tau_r - t^* \le h$, it holds that  
\[
|X(t^*, \omega) - X(\tau_r, \omega)| \le \delta_1 \qquad \forall \omega \in A
\]
and
\begin{equation*}\label{eq:x-star-dist-bound1.5}
d(X(t^*, \omega), \partial D_r) \le \delta_1 \qquad \omega \in A \cap \{\tau_r <T\}.
\end{equation*}

Using Proposition~\ref{prop:conv-rate-strong} for $\gamma =1.5$, we proceed to bound the distance between the exact diffusion process and the numerical solution at time $t^*$:  Let $p^* \in \bN$ be sufficiently large so that $p^*\xi > 1.5$. Then by Proposition~\ref{prop:conv-rate-strong},
\begin{equation}\label{eq:dist-numerical-path-true-path-order1.5}
  \bP(|\barX(t^*) - X(t^*)| \ge r)
  \le \frac{\E{|\barX(t^*) - X(t^*)|^{p^*}}}{h^{(3/2-\xi) p^*}} = \cO(h^{3/2}). 
\end{equation}
We conclude that for $\widetilde{A} := \{\omega \in A \mid |\barX(t^*) - X(t^*)| < r\} $, it holds that $\bP(\widetilde{A}) = 1 - \cO(h^{3/2})$. 

We will next show that $d(\barX(t^*, \omega), \partial D_r) < \delta_2+r$
for all $\omega \in \widetilde{A}\cap \{\tau_r <T\}$, 
which by the time-stepping ~\eqref{eq:adaptive-order-1.5} implies the that 
the smallest step size is used at time $t^*$ in the numerical solution:  
\begin{equation}\label{eq:timestep-tildeA-order1.5}
\Delta t(\barX(t_*, \omega)) = h^3 \quad \forall \omega \in \widetilde{A} \cap \{\tau_r <T\}.
\end{equation}
Observe first that  
\[
d(\barX(t^*, \omega), \partial D_r) \le d(\barX(t^*, \omega), X(t^*, \omega)) + d(X(t^*,\omega),\partial D_r) < \delta_1 +r \quad \forall \omega \in \widetilde{A} \cap \{\tau_r <T\},
\]
and let $\widetilde D:= D \setminus V_{\partial D}(0,\delta_1)$. 
Since $d(\widetilde D, \partial D_r) \ge \delta_1 +r$, we conclude that 
$ \barX(t^*) \notin \widetilde D$ and $\Delta t(\barX(t^*)) \le h^2$
for all paths in $\widetilde{A} \cap \{\tau_r <T\}$.
A recursive argument, where we restrict ourselves to paths in $\widetilde{A} \cap \{\tau_r <T\} \subset A$,
will sharpen the step size estimate to~\eqref{eq:timestep-tildeA-order1.5}: 
The property $\Delta t(\barX(t^*)) \le h^2$ and 
Lemma~\ref{lem:strides-order1.5} implies that
\[
\sup_{s \in [0,\Delta t(\barX(t^*))]} |X(t^{*}+s) - X(t^*)| \le \delta_2. 
\]
It therefore holds that 
\[
\tau_r - t^* \le \Delta t(\barX(t^*)) \le h^2,
\]
which implies that $d(X(\tau_r),X(t^*)) \le \delta_2$ and 
\[
d(\barX(t^*), \partial D_r) \le d(\barX(t^*), X(t^*)) + d(X(t^*), X(\tau_r)) < \delta_2 +r.
\]
This implies that $\barX(t^*) \notin D \setminus V_{\partial D}(0,\delta_2)$, and thus verifies Property~\eqref{eq:timestep-tildeA-order1.5}.

Thanks to~\eqref{eq:timestep-tildeA-order1.5}, we can sharpen the estimate of the distance 
between $\barX(t^*)$ and $X(\tau_r)$: 
\[
  \begin{split}
  \bP( \{|&\barX(t^*) - X(\tau_r)| \ge r \} \cap \{\tau_r <T\}) \\ &\le
  \bP( \{|\barX(t^*) - X(\tau_r)| \ge h^{3/2-\xi} \} \cap \{\tau_r <T\} \cap \widetilde{A}) + \bP(\widetilde{A}^C)\\
  &\le 
  \frac{\E{|\barX(t^*) - X(\tau_r)|^{p^*} \ind{\widetilde{A}\cap \{\tau_r <T\}}}}{ h^{(3/2-\xi)p^*} } + \cO(h^{3/2})\\
  &\le \frac{\E{ p^*|\barX(t^*) - X(t^*)|^{p^*} \ind{\widetilde{A} \cap \{\tau_r <T\}} + p^*|X(t^*) - X(\tau_r)|^{p^*}\ind{\widetilde{A}\cap \{\tau_r <T\}} }}{ h^{(3/2-\xi)p^*} }
  +\cO(h^{3/2}) \\
  &= \cO(h^{3/2}).
  \end{split}
\]
The first summand in the last inequality is bounded by~\eqref{eq:dist-numerical-path-true-path-order1.5}, 
and the second one is bounded by Proposition~\ref{prop:errorInterpolation},~\eqref{eq:timestep-tildeA-order1.5}
and
\[
  \begin{split}
  \E{|X(t^*) - X(\tau_r)|^{p^*}\ind{\widetilde{A}\cap \{\tau_r <T\}}} &\le
  \E{ \max_{k\in \{0,1,\ldots, \lceil T/h^3\rceil-1\}} \sup_{s \in [0,h^3] }
    |X(kh^3 + s) - X(k h^3)|^{p^*} \ind{\widetilde{A}} }\\
  &= \cO\big( h^{3p^*/2} \sqrt{\log(h^{-1})} \big)\,.
  \end{split}
\]
For $G := \{\omega \in \Omega \mid |\barX(t^*) - X(\tau_r)| < r \}$, 
we obtain that $\bP(G^C \cap \{\tau <T\}) = \cO(h^{3/2})$ and  
\[
  \begin{split}
    \omega \in G \cap \{\tau_r < T\} 
    \implies \barX(t^*,\omega) \notin D
    \implies \nu(\omega) \le t^*(\omega)
    &\implies \nu(\omega) \le \tau_r(\omega).
  \end{split}
\]
We conclude the proof by the following observation
\begin{align*}
  II_2 &\le  \E{ (\nu - \tau_r) \ind{\{\nu > \tau\} \cap \{\tau_r < T\}}} \\
  &\le  \underbrace{\E{ (\nu - \tau_r) \ind{\{\nu > \tau\} \cap \{\tau_r < T\}\cap G}}}_{ \le 0} + T \, \bP(G^{C} \cap \{\tau < T\} )= \cO(h^{3/2}).
 \end{align*}
\end{proof}

Up next, we prove the computational cost result for the order 1.5 method.

\begin{proof}[Sketch of proof of Theorem \ref{thm:cost-order1.5}]
Let $s_n := n h^3$ for $n=0,1,\ldots$ denote a set of deterministic uniformly spaced mesh points. This mesh contains all realizations of the adaptive mesh, meaning that $\mesh{\Dt}(\omega)\subset \{s_n\}_{n\ge 0}$ for all $\omega \in \Omega$. Similarly as in~\eqref{eq:costOrder1}, we obtain that   
\begin{equation}\label{eq:costOrder1.5} 
\begin{split}
    \E{\cost{\barX}} &\le  \frac{T}{h} 
    +  \sum_{n=0}^{h^{-3}-1} \frac{h^3}{h^2}\bP\big(\{\nu>s_{n}\} \cap \{ \Dt(\barX(s_n)) = h^2 \}\big) \\
    & \quad + 
     \sum_{n=0}^{h^{-3}-1} \bP \big(\{\nu>s_{n}\} \cap \{ \Dt(\barX(s_n)) = h^3 \}\big).
\end{split}
\end{equation}
Let $A$ denote the maximal stride set in Lemma~\ref{lem:strides-order1.5} 
and let 
\[
  B_i = \{ \omega \mid \max_{t_k \in \mesh{\Dt}} |X(t_k) - \barX(t_k)| \le \delta_i \}, \quad i =1,2.  
\]
If $\omega \in A \cap B_1$, we obtain by similar reasoning as in the proof of Theorem~\ref{thm:cost-order1} that 
\[
  A\cap B_1 \cap \{\nu>s_{n}\} \cap \{ \Dt(\barX(s_n)) = h^2 \} \subset \{\tau_{2\delta_1} > s_n\} \cap \{ d(X(s_n),\partial D) \le 2\delta_1 \},
\]
\[
  A\cap B_2 \cap \{\nu>s_{n}\} \cap \{ \Dt(\barX(s_n)) = h^3 \} \subset \{\tau_{2\delta_2} > s_n\} \cap \{ d(X(s_n),\partial D) \le 2\delta_2 \},
\]
and 
\[
  \bP(A\cap B_1) = 1 - \cO(h^2) \quad \text{and}\quad \bP(A\cap B_2) = 1 - \cO(h^2).
\]
This leads to
\begin{equation}\label{eq:costOrder1.5_2} 
  \begin{split}
      \E{\cost{\barX}} &\le  \cO(h^{-1}) 
      +  \sum_{n=0}^{h^{-3}-1} \frac{h^3}{h^2}\bP \big(\{\tau_{2\delta_1}>s_{n}\} \cap \{ d(X(s_n),\partial D) \le 2\delta_1\} \big) \\
      & \quad + 
       \sum_{n=0}^{h^{-3}-1} \bP \big(\{\tau_{2\delta_2}>s_{n}\} \cap \{ d(X(s_n),\partial D) \le 2\delta_2 \} \big).
  \end{split}
  \end{equation}
  By a similar argument as in the proof of Theorem~\ref{thm:cost-order1},
  it holds that    
\[
  \bP \big(\{\tau_{2\delta_1}>s_{n}\} \cap \{ d(X(s_n),\partial D) \le 2\delta_1\} \big)  = \cO(\delta_1^{2}) = \cO(h \log(h^{-1}))
\]
and 
\[
  \bP \big(\{\tau_{2\delta_2}>s_{n}\} \cap \{ d(X(s_n),\partial D) \le 2\delta_2 \} \big)  
  = \cO(\delta_2^{2}) =  \cO(h^{2} \log(h^{-1})), 
\]
so that 
\begin{align*}
    \E{\cost{\barX}} &= \cO(h^{-1} + h^{-3} \times h^2 \log(h^{-1})) 
    = \cO(h^{-1} \log(h^{-1})).
\end{align*}

\end{proof}

    \section{Numerical experiments}\label{sec:numerics}
    We run several simulations to numerically verify the
    theoretical rates on strong convergence and computational cost for
    the order 1 and 1.5 methods. Algorithm
    \ref{alg:order_gamma_method} describes the implementation of the
    two methods for computing the stopping time of one SDE path. In
    all of the problems below, we consider exit
    times~\eqref{eq:exit-time} with cut-off time $T=10$.
    \begin{algorithm}[h!]
        \begin{algorithmic}[1]
            \Require step size parameter $h$, initial state $x_{0}$, domain $D$, cut-off time $T$ 
            \State Initialize the numerical solution, i.e. $\barX(0) = x_{0} \in D$, set $n=0$ and $t_n=0$.
            \Repeat
                \State Based on whether $\gamma = 1$ or $\gamma = 1.5$, use equation \eqref{eq:adaptive-order-1} or equation \eqref{eq:adaptive-order-1.5},  respectively, to determine the time-step size 
                $\Delta t_n = \Dt(\barX(t_n))$. 
                \State Generate the independent Wiener increment $\Delta W_{n}$ or the tuple of correlated random variables $(\Delta W_{n}, \Delta Z_{n})$ corresponding to the order $\gamma$ of the adaptive method.
                \State Compute the new time $t_{n+1} = t_{n} + \Dt_{n}$, the new state of the associated numerical process 
                \begin{equation*}
                    \barX(t_{n+1}) = \Psi_{\gamma}(\barX(t_{n}), \Dt_{n}), 
                \end{equation*}
                and set $n= n+1$.
            \Until{$\barX(t_{n})$ exits the domain $D$ or until $t_{n} \geq T$, whichever occurs \\ first.} \\
            \Return The exit time of the trajectory: $\nu = t_{n} \wedge T$.
        \end{algorithmic}
        \caption{Order $\gamma$ adaptive time-stepping method}
        \label{alg:order_gamma_method}
    \end{algorithm}
    
    \subsection{Geometric Brownian Motion (GBM)}
    \label{subsec:numerics_1d_gbm}
    To begin, we will investigate the exit time of one-dimensional Geometric Brownian Motion (GBM) from the interval $D = (1, 7)$ for the GBM problem
    \begin{align*}
        \rdX &= 0.05 X \rdt+ 0.2 X\rdW \\
        X(0) &= 4.
    \end{align*}
    
    The reference solution to the mean exit time was computed by
    numerically solving the Feynman--Kac PDE,
    cf.~Proposition~\ref{prop:fk1}, using the Crank--Nicolson method
    for time discretization and continuous, piecewise linear finite
    elements for spatial discretization. To this end, we use the
    \texttt{Gridap.jl} library \cite{Verdugo2022,Badia2020} in the
    \texttt{Julia} programming language. The reference solution for
    the mean exit time of the process starting at $\barX(0) = 4$ is
    $\E{\tau} = 7.153211$, rounded to 7 significant digits. For all
    $l \in \{ 3, 4, 5, 6, 7 \}$, let $\nu_{l}$ represent the exit time
    of the numerical solution $\barX$ from the domain $D$ using the
    step size parameter $h = 2^{-l}$. We estimate the sample moments
    using $M = 10^7$ Monte Carlo samples.
    
    From Figure \ref{fig:gbm_1d}, we observe that the strong
    convergence rate obtained from the numerical simulations agrees
    with the theory for the order 1 and order 1.5 methods and the weak
    convergence rate coincides with the strong rate.  In our numerical
    studies, realizations on neighboring resolutions $\nu_{\ell-1}$
    and $\nu_\ell$ are pairwise coupled using the technique for
    non-nested meshes introduced in~\cite{gilesnonnested2016} together
    with the procedure~\cite[Example 1.1]{hoel2019central} for
    coupling all driving noise in the order 1.5 method. This reduces
    the variance of samples of differences $\nu_l - \nu_{l-1}$ in the
    Monte Carlo estimators for the weak ans strong errors, leading to
    efficient Monte Carlo estimators. We also observe that the expected
    computational complexity involved in implementing the order 1 and
    order 1.5 methods are $\cO(h^{-1} \abs{\log(h)})$ as shown in
    theory, albeit the order 1.5 method is more expensive than the
    order 1 method by a constant.
    \begin{figure}[h!]
        \centering
        \includegraphics[width=0.47\textwidth]{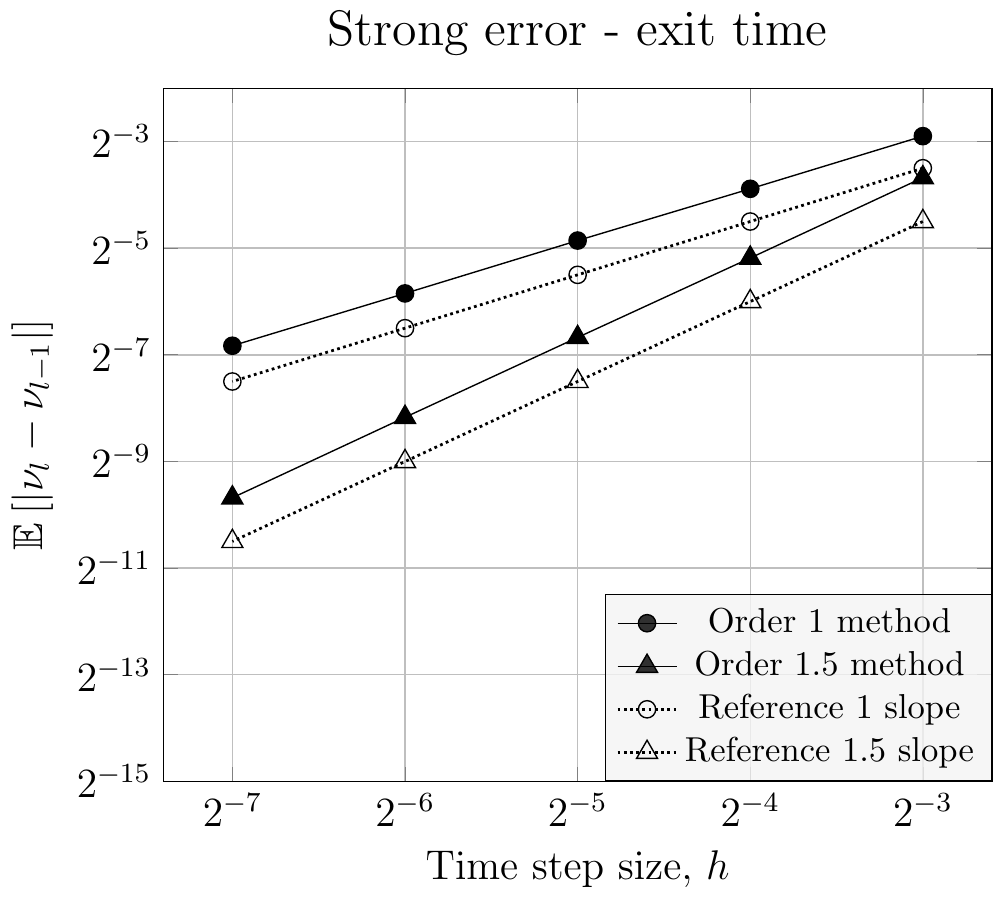}
        \includegraphics[width=0.47\textwidth]{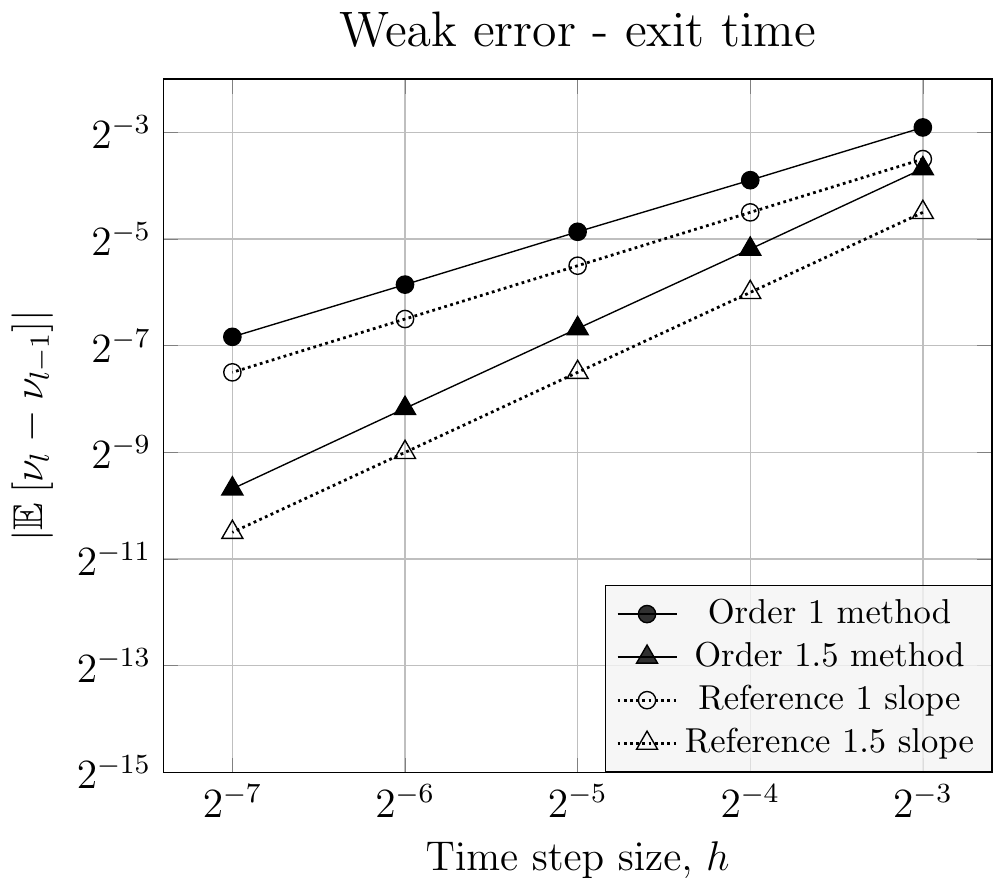}
        \centering
        \includegraphics[width=0.47\textwidth]{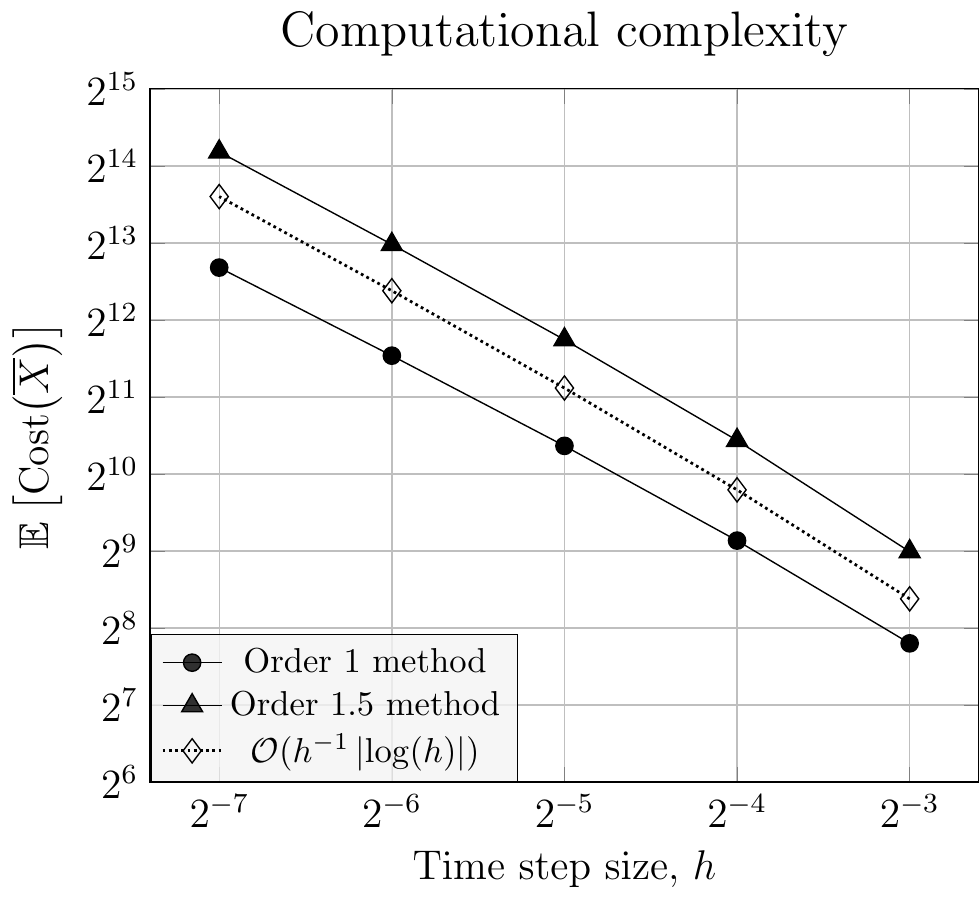}
        \includegraphics[width=0.47\textwidth]{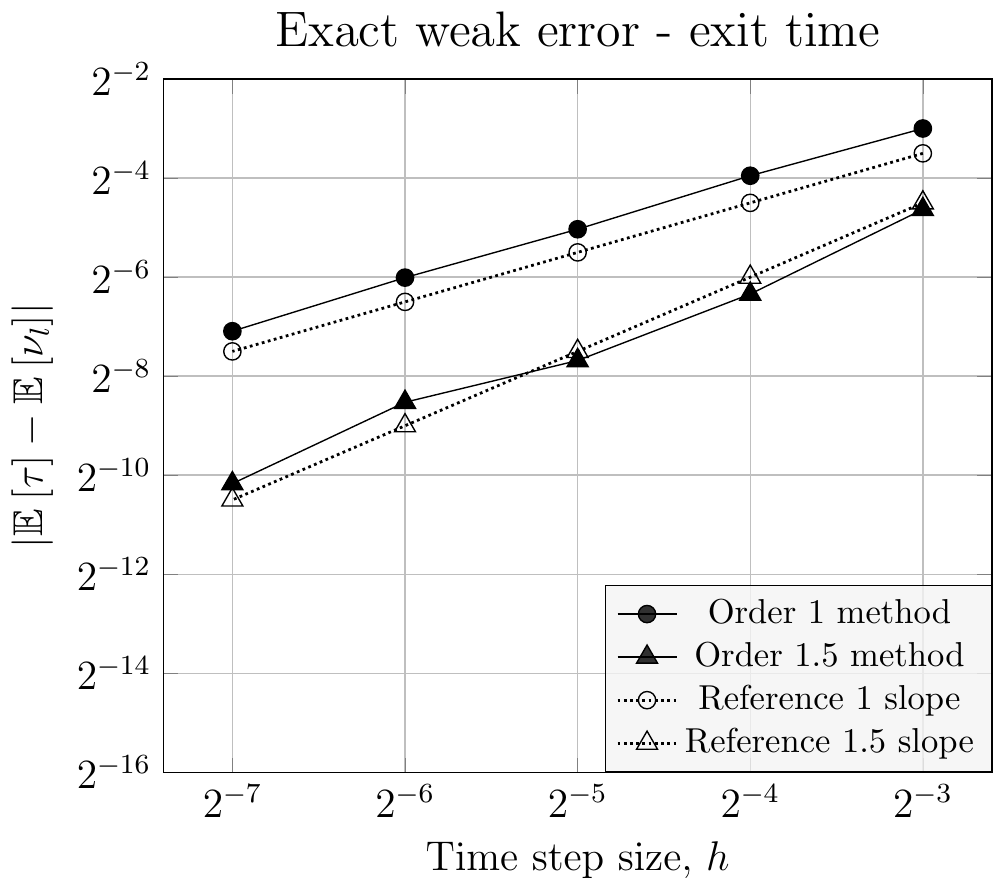}
        \caption{Error convergence rates and total computational complexity involved
          in implementing order 1 and order 1.5 adaptive methods for computing the
          exit time of the GBM process defined in Section \ref{subsec:numerics_1d_gbm}.}
        \label{fig:gbm_1d}
    \end{figure}
    
    \subsection{Linear drift, cosine diffusion coefficient - 1D}
    \label{subsec:numerics_1d_cosine}
    
    For the SDE 
    \begin{align*}
        \rdX &= 0.1 X \,\rdt+ 0.3 \big(\cos(X) + 3 \big) \,\rdW \\
        X(0) &= 4, 
    \end{align*}
    we study the exit time from the interval $D = (1, 7)$. The
    reference solution to the mean exit time problem was computed by
    solving the Feynman--Kac PDE using the same numerical method as in
    the preceding example, yielding $\E{\tau} = 5.504741$, rounded to
    7 significant digits. We run $M = 10^{7}$ simulations for our
    Monte Carlo estimates using the time-step parameter
    $h = {2^{-3}, 2^{-4}, 2^{-5}, 2^{-6}, 2^{-7}}$. From Figure
    \ref{fig:cosine_1d}, the rates for strong error the computational
    cost are in close agreement with theory, and we observe that for
    the given sample size and range of $h$-values, the rate for the
    weak error is more reliably estimated by the sample mean of
    pairwise coupled realizations $\nu_l - \nu_{l-1}$ than the sample
    mean of $\E{\tau} -\nu_\ell$. The reason is that due to the pairwise coupling, the variance of
    $\nu_l - \nu_{l-1}$ is smaller than that of $\E{\tau} -\nu_l$.
    \begin{figure}[h!]
        \centering
        \includegraphics[width=0.47\textwidth]{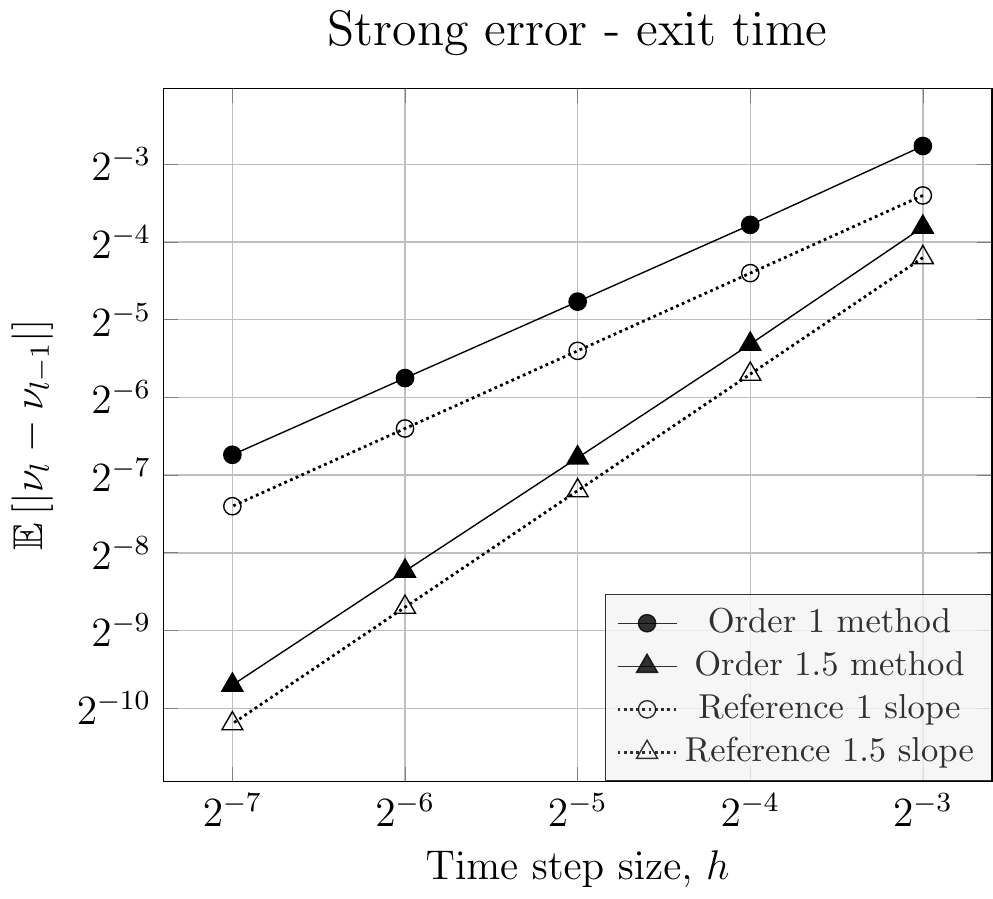}
        \includegraphics[width=0.47\textwidth]{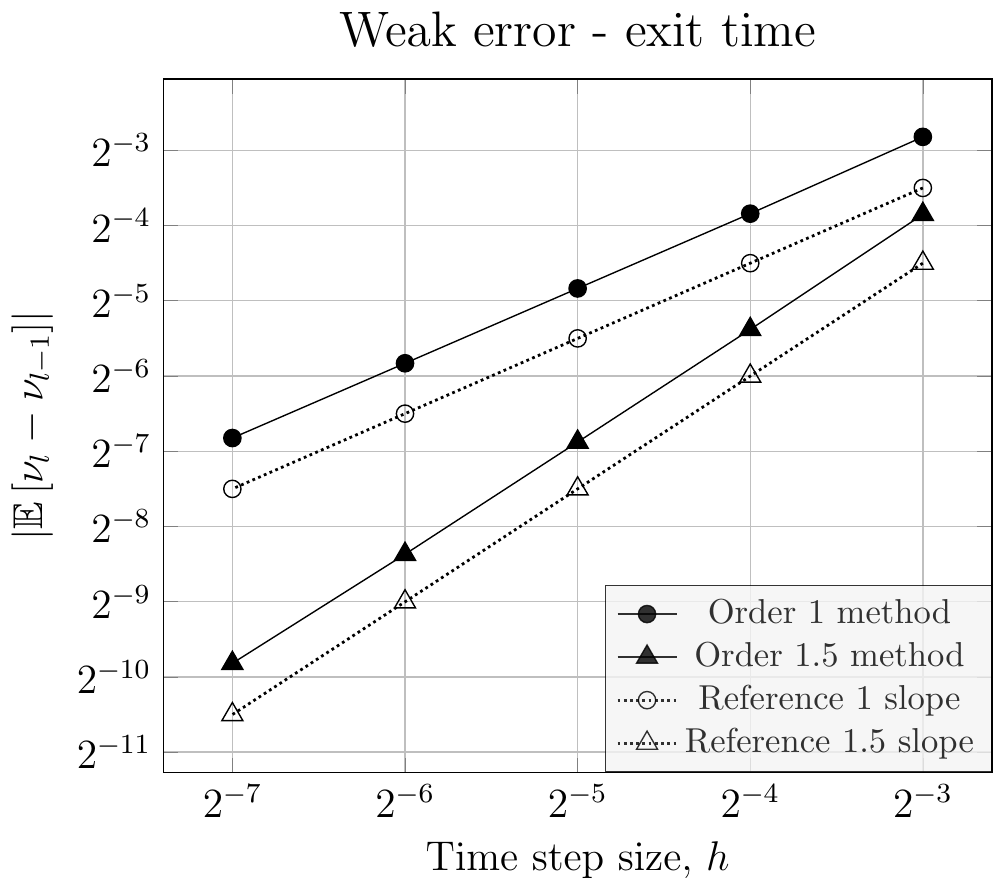}
        \centering
        \includegraphics[width=0.47\textwidth]{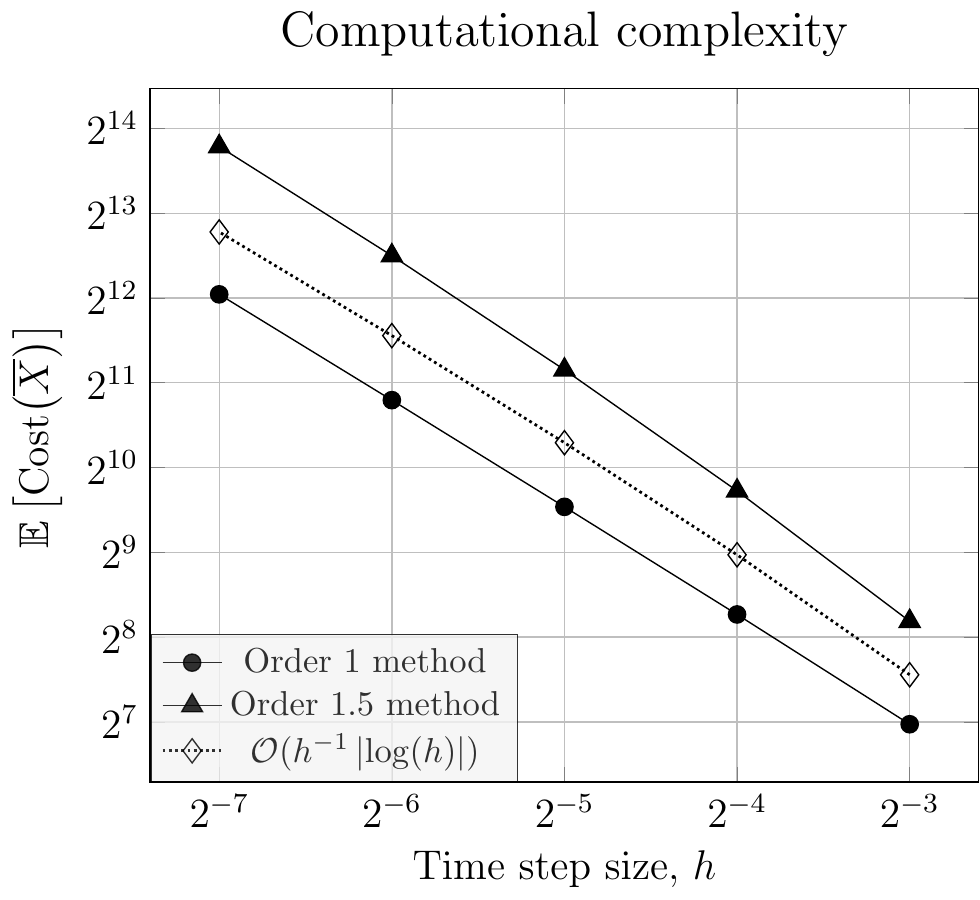}
        \includegraphics[width=0.47\textwidth]{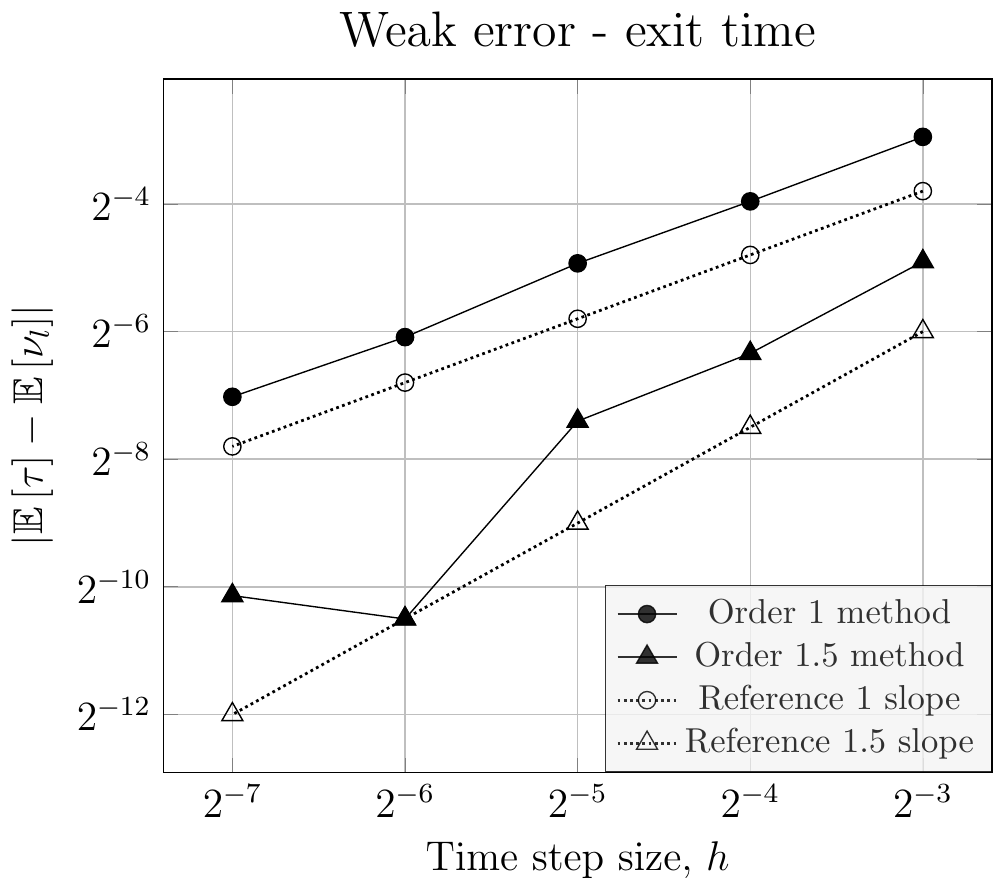}
        \caption{Error convergence rates and total computational complexity involved in implementing
          order 1 and order 1.5 adaptive methods for computing the exit time of SDE with linear drift
          coefficient and non-linear diffusion coefficient in Section \ref{subsec:numerics_1d_cosine}.}
        \label{fig:cosine_1d}
    \end{figure}
    
    We next consider two higher-dimensional exit time problems. 
    
    \subsection{Linear drift, linear diffusion - 2D}
    \label{subsec:numerics_2d_gbm}
    We consider the SDE
    \begin{align*}
        \rdX_{1} &= 0.05 X_{2} \,\rdt+ 0.2 X_{1} \,\rdW^1\\
            \rdX_{2} &= 0.05 X_1 \,\rdt+ 0.2 X_2 \,\rdW^2
    \end{align*}
    with initial condition $  X(0) = (3, 3)^{\top}$, and we are interested in computing the exit time from the disk 
    \begin{align}
        \label{eq:circular_2d_domain}
        D \coloneqq \{ x \in \bR^{2} \; \big| \; \sqrt{(x_{1} - 3)^{2} + (x_{2} - 3)^{2}} < 3 \}\, .
    \end{align}
    The reference solution to the mean exit time problem is computed
    by numerically solving the Feynman--Kac PDE, yielding
    $\E{\tau} = 6.7737$, rounded to 5 significant digits. We
    estimate the sample moments using $M = 10^7$ samples
    for $h = 2^{-3}, 2^{-4}, 2^{-5}, 2^{-6}, 2^{-7}$. Figure
    \ref{fig:gbm_2d} shows that the rates for the strong error and
    computational cost match those from theory and that for the given
    sample size, the weak error rate is more reliably estimated
    by samples of $\nu_l - \nu_{l-1}$ than by samples of $\E{\tau} -\nu_l$.
    
    \begin{figure}[h!]
      \centering
        \includegraphics[width=0.47\textwidth]{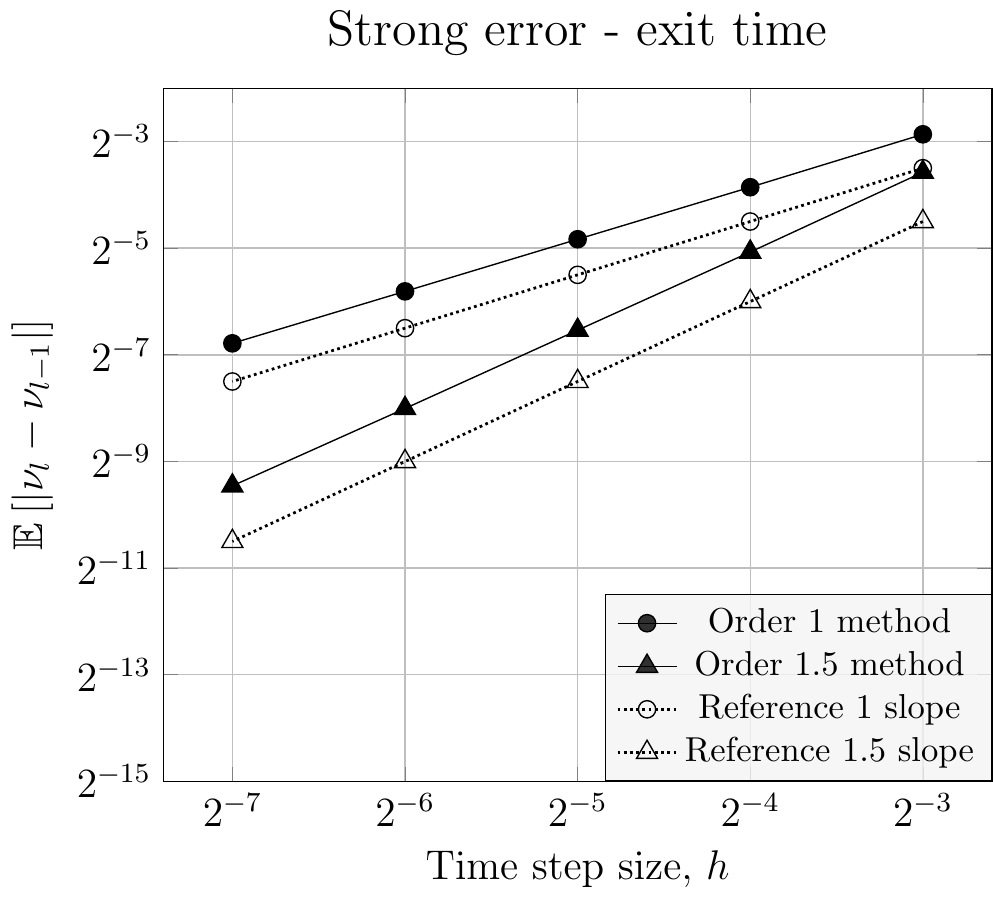}
        \includegraphics[width=0.47\textwidth]{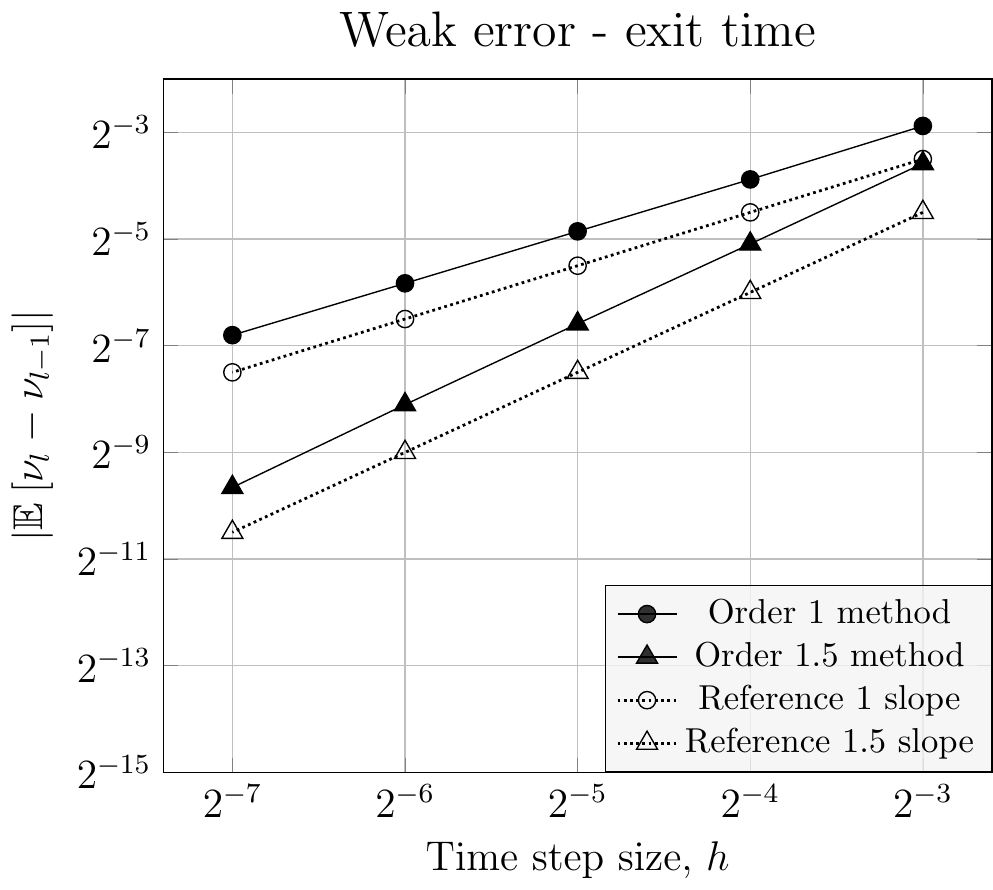}
        \centering
        \includegraphics[width=0.47\textwidth]{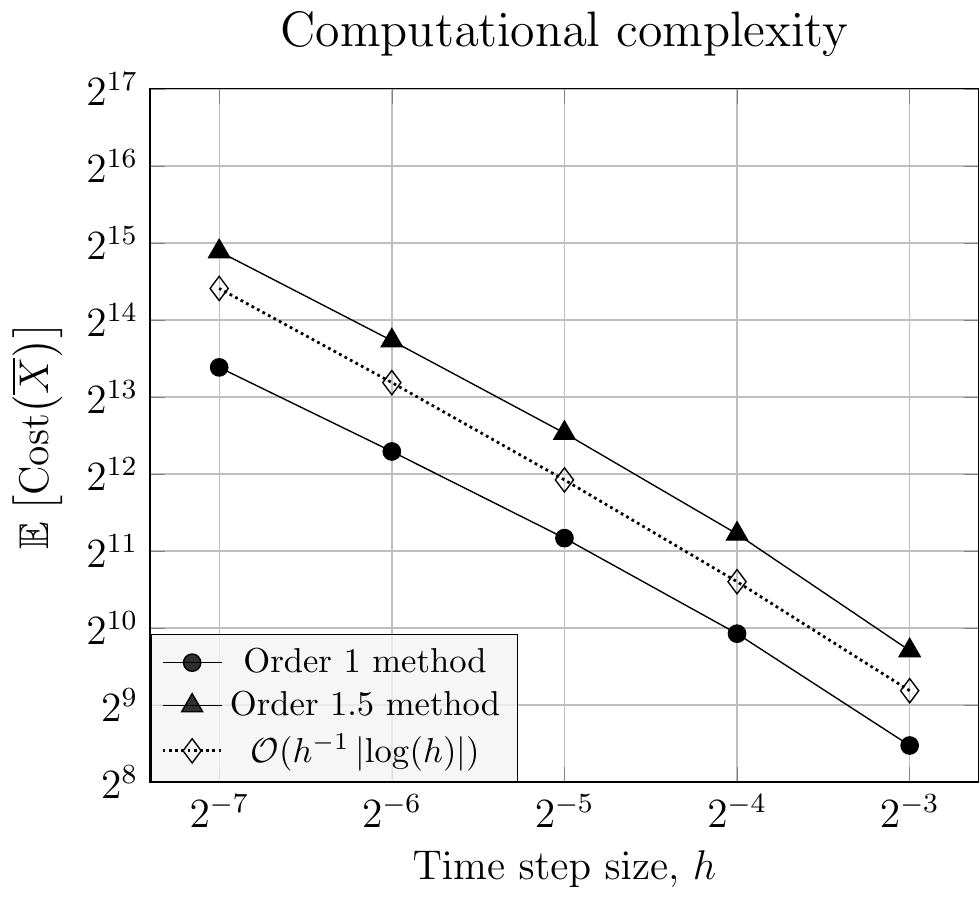}
        \includegraphics[width=0.47\textwidth]{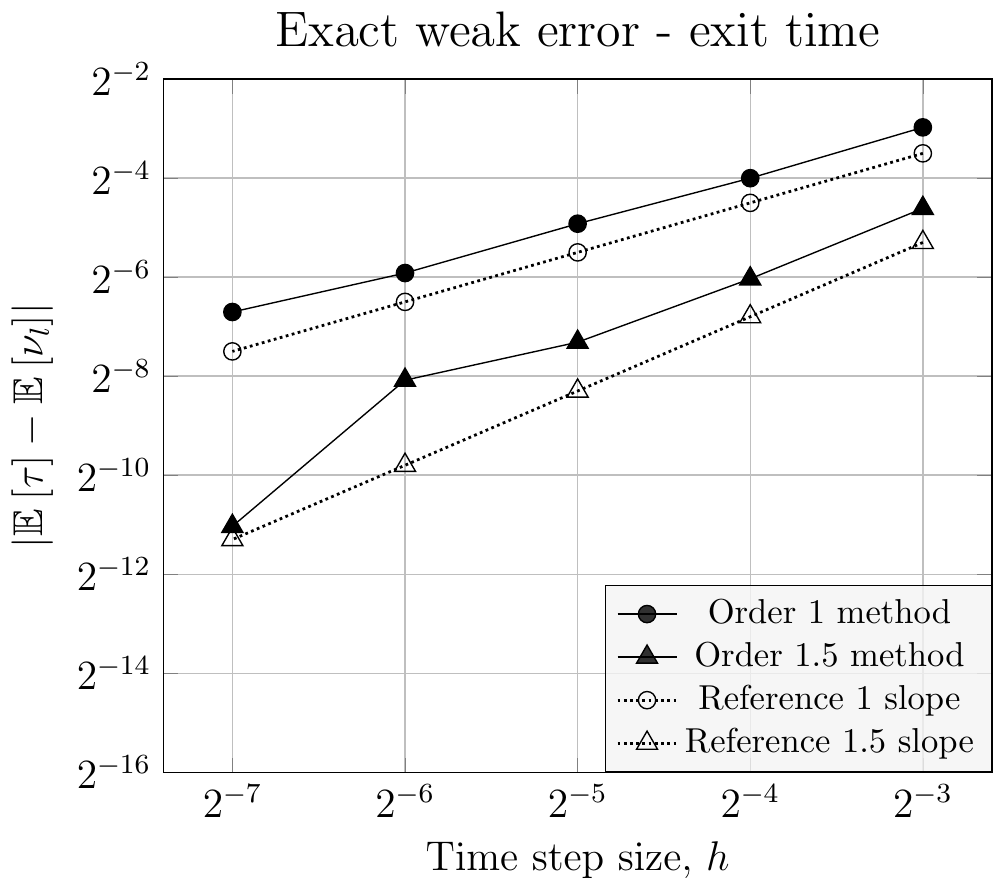}
        \caption{Error convergence rates and total computational cost involved in implementing order
          1 and order 1.5 adaptive methods for computing the exit time of SDE with linear drift
          coefficient and linear diffusion coefficient in Section~\ref{subsec:numerics_2d_gbm}.}
        \label{fig:gbm_2d}
    \end{figure}
    
    \subsection{Linear drift, cosine diffusion - 2D}
    \label{subsec:numerics_2d_cosine}
    We consider the following two-dimensional SDE with linear drift and non-linear, diagonal diffusion coefficient:
    \begin{align*}
        \rdX_{1} &= 0.1 X_{2} \,\rdt+ 0.25 \big( \cos(X_{1}) + 3 \big) \,\rdW^1\\
            \rdX_{2} &= 0.1 X_1 \,\rdt+ 0.25 \big( \cos(X_{2}) + 3 \big) \,\rdW^2,
    \end{align*}
    and with initial condition $X(0) = (3, 3)^{\top}$.  Note that the
    components of the SDE are coupled through the drift term,
    similarly as for the SDE in
    Section~\ref{subsec:numerics_2d_gbm}. We compute the exit
    time of the SDE from the disk domain given by
    equation~\eqref{eq:circular_2d_domain}. The reference solution to
    the mean exit time problem is is computed by numerically solving
    the Feynman--Kac PDE, yielding $\E{\tau} = 5.0853$, rounded to 5
    significant digits.
    The weak and strong errors are estimated using
    $M = 10^{7}$ samples for $h \in [2^{-3}, 2^{-7}]$, and, to reach
    the asymptotic regime of the computational cost, we have estimated
    it over the a range of smaller values, $h \in [2^{-4}, 2^{-14}]$,
    using $M=10^3$ (due to the bump in cost, and, luckily, cost
    estimates do not appear very sensitive to the sample size).
    Figure \ref{fig:cosine_2d} shows that numerical results support
    theory and that for the given sample size, the weak error rate is
    more reliably estimated by the samples of $\nu_l -\nu_{l-1}$ than
    samples of $\E{\tau} - \nu_l$.  Note also that the asymptotic
    regime for the computational cost is not reached at $h=2^{-7}$
    for neither of the methods, but that the cost is proportional
    to $h^{-1} |log(h)|$ for smaller $h$-values.
    \begin{figure}[h!]
      \centering
      \includegraphics[width=0.47\textwidth]{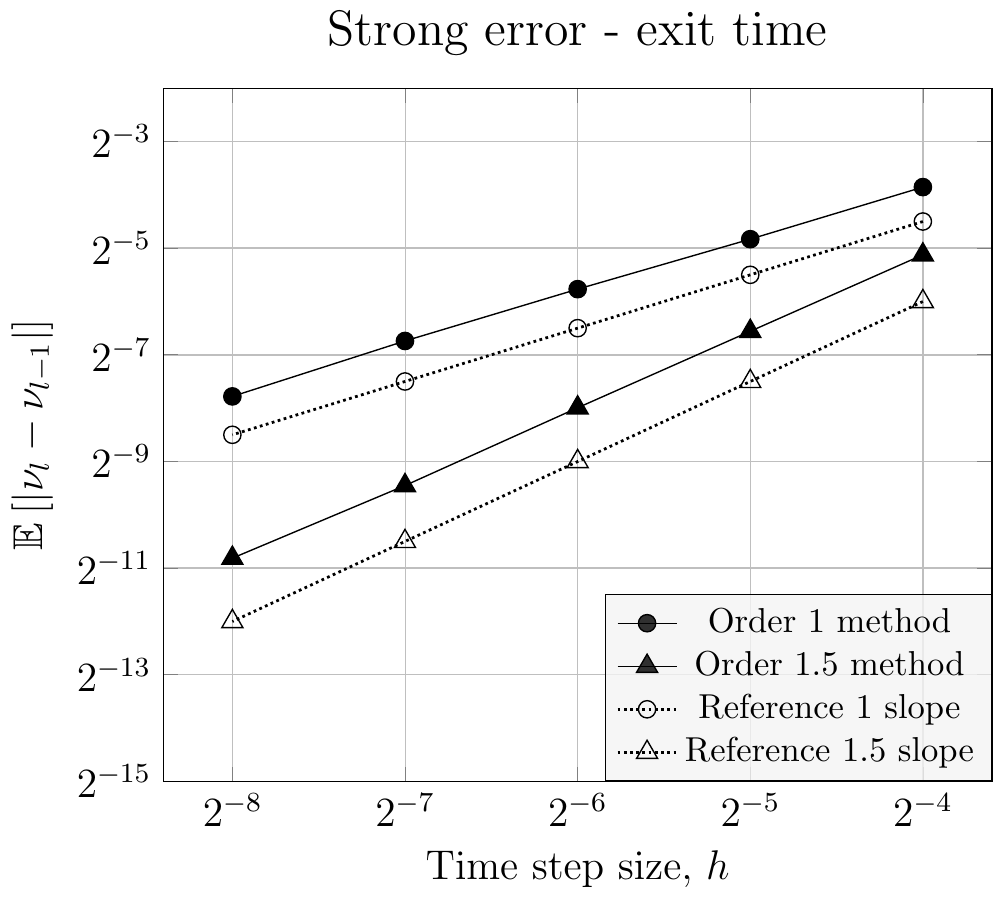}
      \includegraphics[width=0.47\textwidth]{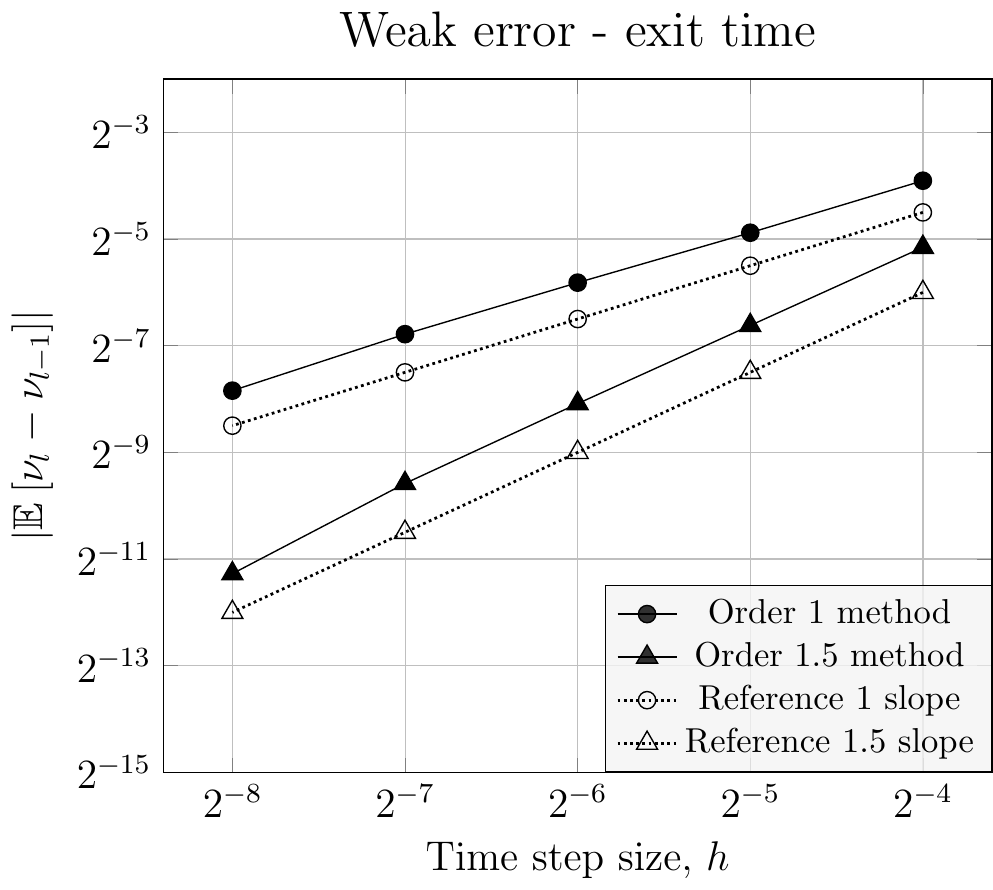}
      \centering
      \includegraphics[width=0.47\textwidth]{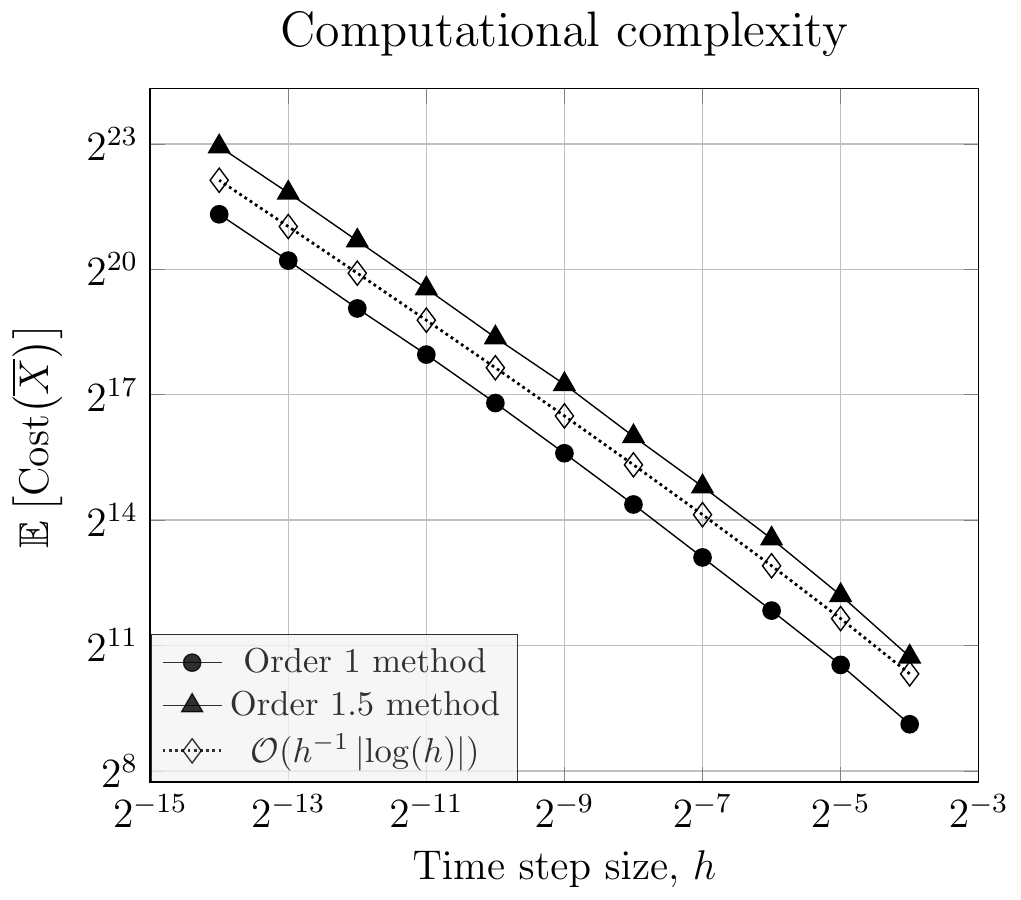}
      \includegraphics[width=0.48\textwidth]{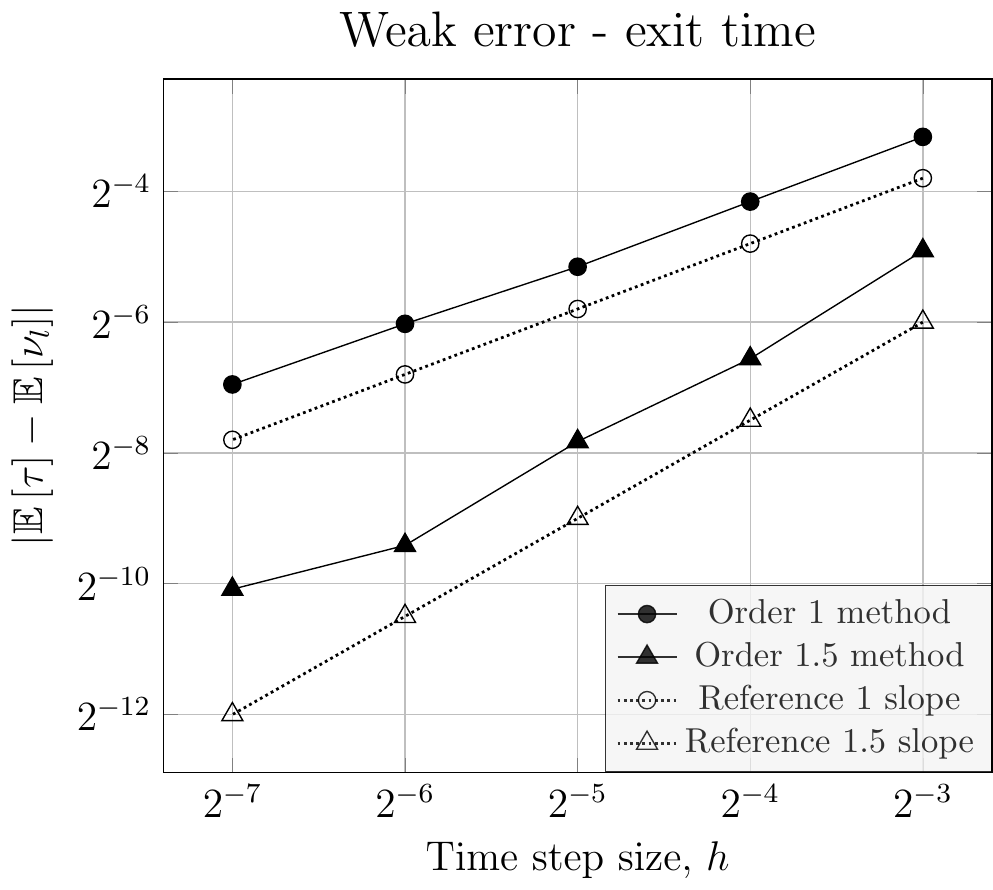}
      \caption{Error convergence rates and total computational cost involved in implementing
        order 1 and order 1.5 adaptive methods for computing the exit time of SDE with linear
        drift coefficient and non-linear diffusion coefficient in Section~\ref{subsec:numerics_2d_cosine}.}
        \label{fig:cosine_2d}
    \end{figure}

    \section{Conclusion}\label{sec:conclusion}
    In this paper, we developed a tractable higher-order adaptive
    time-stepping method for the strong approximation of exit times of
    It\^o-diffusions. We theoretically prove that the Milstein scheme
    combined with adaptive time-stepping with two different step sizes
    lead to a strong convergence rate of $\cO(h^{1-})$, and that the
    strong order 1.5 It\^o--Taylor scheme combined with adaptive
    time-stepping with three different step sizes lead to a strong
    convergence rates of $\cO(h^{3/2-})$. We also showed that the
    expected computational cost for both methods are bounded by
    $\cO(h^{-1} \abs{\log(h)})$. The fundamental idea in our
    approach, recurring in both of the aforementioned methods, is to
    use smaller step sizes as the numerical solution gets closer to
    the boundary of the domain, and to use higher-order
    integration schemes for better approximation of the state of the
    It\^o-diffusion. This reduces both the magnitude of overshoot
    when the numerical solution exits the domain and the probability
    of the numerical solution missing an exit of the domain by the exact
    process.
    
    There are several interesting ways to extend the current work. One
    direction would be to consider strong It\^o--Taylor schemes of
    order $\gamma > 3/2$ combined with adaptive time-stepping that
    employs more than two critical regions, i.e. for some
    $h \in (0, 1)$, we consider time-step sizes finer than $h^{3}$
    when the numerical solution is very close to the boundary
    $\partial D$. This way, the strong convergence rate can be further
    improved while maintaining a computational complexity that is
    tractable.
    
    Although the higher-order adaptive method has been devised for
    improving the strong convergence rate of exit times of
    It\^o-diffusions, the results easily extend to weak approximations
    of quantities of interest (QoI) that depend on the exit time and
    the state of the process:
    \begin{equation*}
        \mathrm{QoI} \coloneqq f(\tau, X_{\tau}) + \int_{0}^{\tau}  g(s, X_{s}) ds,
    \end{equation*}
    where we assume that $f, g$ are functions that are sufficiently smooth.
    
    In the implementation of the higher-order adaptive time-stepping
    algorithm, we use smaller time-step sizes as the numerical process
    gets closer to the boundary of the domain which significantly
    increases the expected computational cost. Combination of
    multilevel Monte Carlo and adaptive
    time-stepping~\cite{hoel2012adaptive,hoel2014implementation,fang2020adaptive,gilesnonnested2016,katsiolides2018multilevel,kelly2018adaptive}
    should significantly reduce the computational cost by using more
    Monte Carlo samples at the coarser level, where we use more
    degrees of freedom larger time-step size, and fewer Monte Carlo
    samples at the finer level, where we use a smaller time-step
    size. To implement the MLMC method, one requires the fine and the
    coarse trajectories of the stochastic process to be strongly
    coupled while also ensuring that the telescoping-sum property is
    satisfied. A smart way to couple solutions of the Euler--Maruyama
    method on non-nested adaptive meshes has been developed
    in~\cite{gilesnonnested2016}, and it would be interesting to study
    how extensible that approach is to higher-order methods, using for
    instance the coupling procedure in~\cite[Example
    1.1]{hoel2019central}.
    
    Finally, in \cite{giles2018multilevel}, a new multilevel Monte
    Carlo methodology to compute the mean exit time or a QoI that
    depends on the exit time of a stochastic process was
    developed. The new MLMC method reduces the multilevel variance by
    computing the approximate conditional expectation once the coarse
    or fine trajectory has exited the domain. Combining our ideas on
    higher-order, adaptive time-stepping with the conditional MLMC
    method has the potential to reduce the computational cost of exit
    time simulations even further. This makes an interesting problem
    for future research.

    \appendix 
    \section{Theoretical results}\label{appendix:contSurfaceArea}

\begin{lemma}\label{lem:gradPBound}
Let Assumptions~\ref{assm:domain} and~\ref{assm:coeffs} hold, and 
for the initial condition $x_0 \in D$ of the SDE~\eqref{eq:SDE} 
let $\lambda :=d(x_0, \partial D)/2$ and let $\bar \delta >0$ be 
sufficiently small such that 
$\bar \delta < \bar R_D/2$ and $d(x_0, \Gamma_{2 \bar \delta})>\lambda$,
where we recall that $\Gamma_{2 \delta} := D_{2\delta}\setminus D_{-2\delta}$
for $\delta \in [0,\bar \delta]$.
Then, there exists a constant $C_p>0$ such that the absorbing-boundary Fokker--Planck equation~\eqref{eq:fokkerPlanckPDE} satisfies 
\[
  \max_{(t,x) \in [0,T]\times \overline \Gamma_{2 \delta}} p(t,x;\delta) \le 4 C_p \delta, \qquad \forall \delta \in [0,\bar \delta],
\] 
where $C_p$ is independent of $\delta$. 
\end{lemma}

\begin{proof}
  Following~\cite[Chapter 3.7]{friedman64}, the solution of the Fokker--Planck equation can be decomposed into the  sum of two functions:
  \[
    p(t,x;\delta) = \widehat \Gamma(t,x;0,x_0) + V(t,x;\delta),
  \]
  where $\widehat \Gamma$ denotes a fundamental solution to a global parabolic extension of the Fokker-Planck PDE from the domain $D_{2 \bar \delta}$ to $\bR^d$ (making it a fundamental solution for $p(t,x;\delta)$ for every $\delta\in [0, \bar \delta]$) 
  and $V$ is a boundary correction term solving the Fokker-Planck equation
  \begin{align*}
    \partial_t V &=  -\nabla \cdot (a V) + \frac{1}{2}  \nabla \cdot \rb{bb^{\top} \nabla V}    && \text{in}  \quad (0,T] \times D_{2\delta}, \\
   V(t,x;\delta)  &= -\widehat \Gamma(t,x;0,x_0)  && \text{on} \quad  (0,T] \times \partial D_{2\delta}\\
   V(0,x;\delta) &= 0 && \quad  x\in \overline D_{2\delta}. 
  \end{align*}
  Assumption~\ref{assm:coeffs} and~\cite[Chapter 1.6]{friedman64}
  implies that $\widehat \Gamma \in C^{1,2}_b([0,T]\times \overline{\Gamma}_{2 \bar \delta})$.
  Let  
  \[
    C_{\widehat \Gamma} := \max_{(x,t) \in [0,T]\times \overline{\Gamma}_{2 \bar \delta}} 
    \widehat \Gamma + |\partial_t \widehat \Gamma| + |\nabla \widehat \Gamma| + |\nabla^2 \widehat \Gamma|,
  \] 
  and note that also the boundary data for $V$,
  \[
  \varphi(t,x;\delta) = 
  \begin{cases}
    -\widehat \Gamma(t,x;0,x_0) & (t,x) \in (0,T]\times \partial D_{2\delta}\\
    0 & (t,x) \in \{0\} \times \overline D_{2\delta}
  \end{cases}
  \]
  satisfies $\varphi \in C^{1,2}_b\Big( \big((0,T]\times \partial D_{2\delta} \big)\cup \big(\{0\} \times \overline D_{2\delta}\big)\Big)$ with the following uniform upper bound 
  over all points on the boundary domain
  \[
    |\varphi| + |\partial_t \varphi| + |\nabla \varphi| + |\nabla^2 \varphi| \le C_{\widehat \Gamma} \qquad \forall \delta \in [0,\bar \delta].  
  \]

  The boundary gradient of $V$ will be bounded 
  using~\cite[Lemma 10.1]{lieberman96}, so we need to verify 
  that the lemma applies to $V$ for any $\delta \in [0,\bar \delta]$. 
  Note first that by the maximum principle, it holds that 
  $|V|\le C_{\widehat \Gamma}$ for all 
  $\delta \in [0,\bar \delta]$. 
  Assumption~\ref{assm:domain} and $\bar R_D < \reach(D)/2$ 
  further implies that $D_{2\delta}$ satisfies the uniform ball condition 
  with $\reach(D_{2\delta}) > \reach(D)/2>\bar R_D$ for all 
  $\delta \in [0,\bar \delta]$. Therefore, any point $x \in \partial D_{2\delta}$
  for any $\delta \in[0,\bar \delta]$ satisfies the infinite exterior cylinder
  condition with radius $\tilde r := \bar R_D$, cf.~\cite[Lemma 10.1]{lieberman96}.
  Thanks to Assumption~\ref{assm:coeffs} there exists a constant 
  $C_{ab}>0$ that depends on the supremum of the SDE coefficents $a$ and $b$ 
  and their first and second partial derivatives on $D_{\bar R_D}$ 
  such that the left-hand side of~\cite[inequality (10.6)]{lieberman96}
  is bounded from above by 
  \[
    g_L(p) := 1+ \Big(\widehat C_b + C_{ab}(C_{\widehat \Gamma} +1) \Big)|p| \quad \text{for} \quad p \in \bR^d
  \]
  while the $\mu$-scaled Bernstein coefficient on the right-hand side
  of the same inequality is bounded from below by $g_R(p; \mu) := \mu \hat c_b |p|^2$, 
  where $\mu>0$ is the scaling parameter and $\widehat C_b>\hat c_b>0$ 
  are defined in Assumption~\ref{assm:coeffs}. This implies that one 
  can find constants $\mu, p_0>0$ that are independent of $\delta$ such that 
  \[
  g_L(p) \le g_R(p; \mu)\qquad \text{for all}\quad  p \in \bR^d \quad \text{such that} \quad |p| \ge p_0.
  \]
  
  We have shown that~\cite[Lemma 10.1]{lieberman96} applies, 
  which means there exists a constant $C_V>0$ that depends on 
  $\mu,p_0$, the uniform infinite exterior cylinder radius $\tilde r$ 
  and $C_{\widehat \Gamma}$ such that 
  \[
  \sup_{(t,x,y) \in [0,T] \times D_{2\delta} \times \partial D_{2\delta}} 
  |V(t,x;\delta) - V(t,y;\delta)| \le C_V |x-y|   
  \]   
  holds for all $\delta \in [0,\bar \delta]$. And since 
  \[
    |p(t,x;\delta) - p(t,y;\delta)| \le |\widehat \Gamma(t,x) - \widehat \Gamma(t,y)| + |V(t,x;\delta) - V(t,y;\delta)|, 
  \]
  and $p(t,y;\delta) =0$ for $y \in \partial D_{2\delta}$, 
  we conclude that 
  \[
  \max_{(t,x) \in [0,T] \times \overline{\Gamma}_{2\delta}} 
  p(t,x;\delta) \le \underbrace{(C_{\widehat \Gamma} + C_V)}_{=:C_p} 
  \max_{x \in \overline{\Gamma}_{2\delta}} d(x,\partial D_{2\delta}) \le 4 C_p \delta.    
  \]

\end{proof}

\begin{lemma}\label{lem:cI_bounded}
  For a domain $D$ satisfying Assumption~\ref{assm:domain} and $R = {reach}(D)/2)$, the mapping $\cI:[-R,R] \to [0,\infty)$ defined by 
  \[
    \cI(r) = \int_{\partial D_{r}} dS(x),
  \]
  satisfies that $\max_{r \in [-R,R]} \cI(r) < \infty$.
\end{lemma}

\begin{proof}
  For each $r \in [-R, R]$ we recall that $\partial D_{r}$ is $C^3$ and there exists a constant $C_r>0$ such that 
  \[
  \cI(r) \le C_r \int_{D_R} dx,
  \]
  cf.~\cite[Lemma 6.36]{lieberman96}.
  To obtain a uniform upper bound we introduce the following mapping on the domain $\overline{\Gamma}_R = \{x \in D_R \mid d(x, \partial D_R) \le 2R \}$:
  \[
  d(x) := 
  \begin{cases}
  d(x,D) & \text{if} \quad x \in D \cap \overline{\Gamma}_{R}\\
  -d(x,D) & \text{if} \quad x \in D^C \cap \overline{\Gamma}_{R}.
  \end{cases}
  \]
  A small extension of~\cite[Lemma 14.16]{gilbarg2015elliptic} yields that $d \in C^3(\overline{\Gamma}_{R})$ and that for any 
  $x \in \partial D_r$ and $r \in [-R, R]$,  
  \[
  - \nabla d(x) = n_{D}(\pi_r^{-1}(x)) = n_{D_r}(x),  
  \]
  where we recall that $n_{D_r}$ denotes the outward pointing normal on the 
  boundary $\partial D_r$. For any $r \in [-R, R]$ the divergence 
  theorem yields 
  \[
    \begin{split}
  |\cI(r) - \cI(0)| &= \left|\int_{\partial D_{r}}  |\nabla d(x)|^2 dS(x) - \int_{\partial D_0} |\nabla d(x)|^2 dS(x) \right|\\ 
  &\le \int_{\Gamma_{R}} |\Delta d(x)| dx < \infty,
    \end{split} 
  \]
  where the last inequality follows from $\Delta d \in C^1(\overline{\Gamma}_{ R})$. We conclude that 
  \[
  \max_{r \in [-R, R]} \cI(r) \le \cI(0) 
  + \max_{r \in [-R, R]}|\cI(r) - \cI(0)| < \infty. 
  \]  
\end{proof}

    \section{Computational cost for the order 1 method - 1D Wiener process}\label{appendix:a}

    In this section we use a more fundamental approach to deriving 
    an upper bound for the computational cost 
    of implementing the order 1 adaptive method for a
    one-dimensional Wiener process $W(t)$ exiting the interval
    $(a, b)$. This is possible thanks to the availability of 
    the joint density for the state of the process $W(t)$ 
    and its running maximum/minimum. 
    
    Let $\barW(t)$ denote the continuous time extension of
    the discretely sampled Wiener process. For $t > s \geq 0$, let
    $M_{W}(s, t)$ and $m_{W}(s, t)$ denote the running maximum and the
    running minimum of the one-dimensional Wiener process $W(t)$,
    respectively, over the interval $[s, t]$, i.e.
    \begin{align*}
        M_{W}(s, t) &\coloneqq \sup_{r \in [s, t]} \big( W(r) - W(s) \big) \\
        m_{W}(s, t) &\coloneqq \inf_{r \in [s, t]} \big( W(r) - W(s) \big) 
    \end{align*}
    Let the shorthands $M_{W}(t)$ and $m_{W}(t)$
    represent the running maximum and the
    running minimum of the Wiener process over the interval $[0, t]$,
    respectively. Recall that the running maximum and the negative of
    the running minimum of the Wiener process have the same
    distribution, cf. \cite[Section 3.6]{klebaner2012introduction}.
    
    The threshold parameter $\delta(h)$, for all $h \in (0, 1)$, is defined as
    \begin{align}
        \label{eq:adpt-order-1-wiener}
        \delta(h) \coloneqq \sqrt{4 h \log (h^{-1})}
    \end{align}
    We construct the set $A$ to bound the strides of the running maximum and the running minimum of the Wiener process
    \begin{equation}\label{eq:aSet_appendixProof}
    \begin{split}
        A &\coloneqq \left\{ \omega \in \Omega \; \big| \; M_{W}(t_{n}, t_{n+1}) \leq \delta, \forall n \in \{ 0, 1, \ldots, N - 1 \} \right\} \\
        & \qquad \cap \left\{ \omega \in \Omega \; \big| \; m_{W}(t_{n}, t_{n+1}) \geq -\delta, \forall n \in \{ 0, 1, \ldots, N - 1 \} \right\}
    \end{split}
  \end{equation}
    It can be shown for all $h < e^{-1}$, in a manner similar to the proof argument for Lemma~\ref{lem:strides-order1},
    and for the choice of $\delta$ given by equation \eqref{eq:adpt-order-1-wiener} that
    \[ \Prob{A^{C}} = \cO(h \abs{\log (h)}^{-1/2}) \]
    
    \begin{proposition}[Computational cost for the order 1 method - 1D Wiener process]
        For a sufficiently small step size parameter $h > 0$, the joint probability density of the tuple of random variables $(W(t), M_{W}(t))$ given by
        \begin{align*}
            \rho_{(W(t), M_{W}(t))}(w, m) \coloneqq \sqrt{\frac{2}{\pi}} \frac{2m - w}{t^{3/2}} \exp \left( -\frac{(2m - w)^{2}}{2t} \right), \quad m \geq 0, \; w \leq m
        \end{align*}
        is well-defined for any $t \in [h, T]$ and 
        \begin{align*}
            \sup_{(t,w,m) \in [h, T] \times [b-\delta, b) \times [b - \delta, b + \delta]} \rho_{(W(t), M_{W}(t))}(w, m) < \infty\, .
        \end{align*}
        It further holds that
        \[ \E{\cost{\barW}} = \cO(h^{-1} \abs{\log(h)})\, . \]
    \end{proposition}
    \begin{proof}
        Based on the size of the strides taken by the Wiener process, the computational cost can be decomposed as follows:
        \begin{align*}
            \E{\cost{\barW}} = \E{\cost{\barW} \ind{A^{C}}} + \E{\cost{\barW} \ind{A}} \eqqcolon I + II,
        \end{align*}
        where the set $A$ was defined in~\eqref{eq:aSet_appendixProof}.
        For term $I$, note that the upper bound on the number of time-steps that can be taken by the Wiener process before it exits the domain, using the order 1 adaptive time-stepping method, is $T/h^{2}$. From this, we obtain
        \begin{align*}
            I &= \E{\cost{\barW} \ind{A^{C}}} \leq \frac{T}{h^{2}} \E{\ind{A^{C}}} = \frac{T}{h^{2}} \Prob{A^{C}} = \cO(h^{-1} \abs{\log(h)}^{-1/2}) \, .
        \end{align*}
        For term $II$, we can write the computational cost as
        \begin{align*}
            II &= \E{\cost{\barW} \ind{A}} 
            \leq \frac{T}{h} + \int_{h}^{T} \frac{\E{\ind{\{ \nu > t \} \cap \{ \Dt(\barW(t)) = h^{2} \} \cap A}}}{h^{2}}\rdt\, .
        \end{align*}
        Note that
        \begin{multline*}
            \{ \nu > t \} \, \mathlarger{\cap} \, \{ \Dt(\barW(t)) = h^{2} \} \, \mathlarger{\cap} \, A \, \mathlarger{\subset}\\
            \Big(\{ M_{W}(t) \in [b - \delta, b + \delta) \} \, \mathlarger{\cap}\,  \{ W(t) \in [b-\delta, b) \} \Big) \, \mathlarger{\cup} \\ \Big(\{ m_{W}(t) \in (a - \delta, a + \delta] \}\, \mathlarger{\cap} \,\{ W(t) \in (a, a + \delta] \} \Big) \, .
        \end{multline*}
       For any $h < \min(e^{-1},T)$, $t \in [h, T]$ and $(w, m) \in [b-\delta, b) \times [b- \delta, b + \delta]$, it holds that
        \begin{align*}
            \rho_{(W(t), M_{W}(t))} (w, m) &= \sqrt{\frac{2}{\pi}} \frac{(2m - w)}{t^{3/2}} \exp \left( -\frac{(2m - w)^{2}}{2 t} \right) \\
            &\leq \sqrt{\frac{2}{\pi}} \frac{b + 3 \delta}{t^{3/2}} \exp \left( -\frac{(b - \delta)^{2}}{2 t} \right) \\
            &\leq \max_{t \in [h, T]} \sqrt{\frac{2}{\pi}} \frac{b + 6 e^{-1/2}}{t^{3/2}} \exp \left( -\frac{(b - 2e^{-1/2})^{2}}{2 t} \right) \\
            &=: C_{\rho}^b < \infty.
        \end{align*}
        Using a similar argument, we can bound the joint probability density function $\rho_{(W(t), m_{W}(t))}(w,m)$ uniformly
        for all $(w,m) \in (a,a+\delta] \times [a-\delta, a+\delta]$ and $t \in [h,T]$ by
        a positive constant $C_{\rho}^{a} < \infty$.
        From the above results, we obtain that
        \begin{align*}
            II &\leq \frac{T}{h} + \frac{C_{\rho}^{a} \delta^{2} (T - h)}{h^{2}} + \frac{C_{\rho}^{b} \delta^{2} (T - h)}{h^{2}} = \cO(h^{-1} \abs{\log(h)}) \, .
        \end{align*}
    \end{proof}

\subsubsection*{Acknowledgments} 

Thanks to Mike Giles for remarks and discussions on the manuscript, particularly for the very nice idea to use the absorbing-boundary Fokker--Planck equation in the proof of the computational cost of the adaptive methods. This improved our computational cost result considerably! 

We would also like to thank Snorre Harald Christensen, Kenneth Karlsen, Peter H.C. Pang and Nils Henrik Risebro for helpful discussions on parabolic and elliptic PDE.

    ~\\

\subsubsection*{Financial support}

This work was supported by the Mathematics Programme of the Trond Mohn Foundation.

    \bibliographystyle{amsplain}
    \bibliography{references.bib}

\providecommand{\bysame}{\leavevmode\hbox to3em{\hrulefill}\thinspace}
\providecommand{\MR}{\relax\ifhmode\unskip\space\fi MR }
\providecommand{\MRhref}[2]{%
  \href{http://www.ams.org/mathscinet-getitem?mr=#1}{#2}
}
\providecommand{\href}[2]{#2}
\begin{thebibliography}{10}

\bibitem{alsmeyer1994markov}
Gerold Alsmeyer, \emph{On the markov renewal theorem}, Stochastic processes and
  their applications \textbf{50} (1994), no.~1, 37--56.

\bibitem{Badia2020}
Santiago Badia and Francesc Verdugo, \emph{Gridap: An extensible finite element
  toolbox in julia}, Journal of Open Source Software \textbf{5} (2020), no.~52,
  2520.

\bibitem{baldi2017stochastic}
Paolo Baldi, \emph{Stochastic calculus}, Stochastic Calculus, Springer, 2017,
  pp.~215--254.

\bibitem{bayer2020pricing}
Christian Bayer, Ra{\'u}l Tempone, and S{\"o}ren Wolfers, \emph{Pricing
  american options by exercise rate optimization}, Quantitative Finance
  \textbf{20} (2020), no.~11, 1749--1760.

\bibitem{bouchard_first_2017}
Bruno Bouchard, Stefan Geiss, and Emmanuel Gobet, \emph{First time to exit of a
  continuous {Itô} process: {General} moment estimates and
  ${\text{\uppercase{l}}}_{1}$-convergence rate for discrete time
  approximations}, Bernoulli \textbf{23} (2017), no.~3, 1631--1662, Publisher:
  Bernoulli Society for Mathematical Statistics and Probability.

\bibitem{broadie1997continuity}
Mark Broadie, Paul Glasserman, and Steven Kou, \emph{A continuity correction
  for discrete barrier options}, Mathematical Finance \textbf{7} (1997), no.~4,
  325--349.

\bibitem{dalphin2014some}
J{\'e}r{\'e}my Dalphin, \emph{Some characterizations of a uniform ball
  property}, ESAIM: proceedings and surveys \textbf{45} (2014), 437--446.

\bibitem{deaconu2017walk}
Madalina Deaconu, Samuel Herrmann, and Sylvain Maire, \emph{The walk on moving
  spheres: a new tool for simulating brownian motion’s exit time from a
  domain}, Mathematics and Computers in Simulation \textbf{135} (2017), 28--38.

\bibitem{fang2020adaptive}
Wei Fang and Michael~B Giles, \emph{Adaptive euler--maruyama method for sdes
  with nonglobally lipschitz drift}, The Annals of Applied Probability
  \textbf{30} (2020), no.~2, 526--560.

\bibitem{friedman64}
Avner Friedman, \emph{Partial differential equations of parabolic type},
  Prentice-Hall, Inc., Englewood Cliffs, N.J., 1964. \MR{0181836}

\bibitem{friedmanVol1}
\bysame, \emph{Stochastic differential equations and applications. {V}ol. 1},
  Probability and Mathematical Statistics, Vol. 28, Academic Press, New
  York-London, 1975. \MR{0494490}

\bibitem{gilbarg1980intermediate}
David Gilbarg and Lars H{\"o}rmander, \emph{Intermediate schauder estimates},
  Archive for Rational Mechanics and Analysis \textbf{74} (1980), no.~4,
  297--318.

\bibitem{gilbarg2015elliptic}
David Gilbarg and Neil~S Trudinger, \emph{Elliptic partial differential
  equations of second order}, vol. 224, springer, 2015.

\bibitem{giles2018multilevel}
Michael~B Giles and Francisco Bernal, \emph{Multilevel estimation of expected
  exit times and other functionals of stopped diffusions}, SIAM/ASA Journal on
  Uncertainty Quantification \textbf{6} (2018), no.~4, 1454--1474.

\bibitem{gilesnonnested2016}
Michael~B. Giles, Christopher Lester, and James Whittle, \emph{Non-nested
  {Adaptive} {Timesteps} in {Multilevel} {Monte} {Carlo} {Computations}}, Monte
  {Carlo} and {Quasi}-{Monte} {Carlo} {Methods} (Cham) (Ronald Cools and Dirk
  Nuyens, eds.), Springer {Proceedings} in {Mathematics} \& {Statistics},
  Springer International Publishing, 2016, pp.~303--314 (en).

\bibitem{Gobet2000-rf}
Emmanuel Gobet, \emph{Weak approximation of killed diffusion using euler
  schemes}, Stochastic Process. Appl. \textbf{87} (2000), no.~2, 167--197.

\bibitem{Gobet2001-pr}
\bysame, \emph{Euler schemes and half-space approximation for the simulation of
  diffusion in a domain}, ESAIM Probab. Stat. \textbf{5} (2001), 261--297.

\bibitem{gobet_exact_2004}
Emmanuel Gobet and Stéphane Menozzi, \emph{Exact approximation rate of killed
  hypoelliptic diffusions using the discrete {Euler} scheme}, Stochastic
  Processes and their Applications \textbf{112} (2004), no.~2, 201--223 (en).

\bibitem{gobet_stopped_2010}
\bysame, \emph{Stopped diffusion processes: {Boundary} corrections and
  overshoot}, Stochastic Processes and their Applications \textbf{120} (2010),
  no.~2, 130--162 (en).

\bibitem{higham_mean_2013}
Desmond~J. Higham, Xuerong Mao, Mikolaj Roj, Qingshuo Song, and George Yin,
  \emph{Mean {Exit} {Times} and the {Multilevel} {Monte} {Carlo} {Method}},
  SIAM/ASA Journal on Uncertainty Quantification \textbf{1} (2013), no.~1,
  2--18, Publisher: Society for Industrial and Applied Mathematics.

\bibitem{hoel2019central}
H{\aa}kon Hoel and Sebastian Krumscheid, \emph{Central limit theorems for
  multilevel monte carlo methods}, Journal of Complexity \textbf{54} (2019),
  101407.

\bibitem{hoel2012adaptive}
H{\aa}kon Hoel, Erik~von Schwerin, Anders Szepessy, and Ra{\'u}l Tempone,
  \emph{Adaptive multilevel monte carlo simulation}, Numerical Analysis of
  Multiscale Computations, Springer, Berlin, Heidelberg, 2012, pp.~217--234.

\bibitem{hoel2014implementation}
H{\aa}kon Hoel, Erik Von~Schwerin, Anders Szepessy, and Ra{\'u}l Tempone,
  \emph{Implementation and analysis of an adaptive multilevel monte carlo
  algorithm}, Monte Carlo Methods and Applications \textbf{20} (2014), no.~1,
  1--41.

\bibitem{katsiolides2018multilevel}
Grigoris Katsiolides, Eike~H M{\"u}ller, Robert Scheichl, Tony Shardlow,
  Michael~B Giles, and David~J Thomson, \emph{Multilevel monte carlo and
  improved timestepping methods in atmospheric dispersion modelling}, Journal
  of Computational Physics \textbf{354} (2018), 320--343.

\bibitem{kelly2018adaptive}
C{\'o}nall Kelly and Gabriel~J Lord, \emph{Adaptive time-stepping strategies
  for nonlinear stochastic systems}, IMA Journal of Numerical Analysis
  \textbf{38} (2018), no.~3, 1523--1549.

\bibitem{klebaner2012introduction}
Fima~C Klebaner, \emph{Introduction to stochastic calculus with applications},
  World Scientific Publishing Company, 2012.

\bibitem{kloeden1992stochastic}
Peter~E Kloeden and Eckhard Platen, \emph{Stochastic differential equations},
  Numerical solution of stochastic differential equations, Springer, 1992,
  pp.~103--160.

\bibitem{lieberman96}
Gary~M. Lieberman, \emph{Second order parabolic differential equations}, World
  Scientific Publishing Co., Inc., River Edge, NJ, 1996. \MR{1465184}

\bibitem{merleProhl2022}
Fabian Merle and Andreas Prohl, \emph{A posteriori error analysis and
  adaptivity for high-dimensional elliptic and parabolic boundary value
  problems}, Numerical Analysis preprint server at University of T\"ubingen
  (2022), \url{https://na.uni-tuebingen.de/preprints.shtml}.

\bibitem{morters2010brownian}
Peter M{\"o}rters and Yuval Peres, \emph{Brownian motion}, vol.~30, Cambridge
  University Press, 2010.

\bibitem{muller1956some}
Mervin~E Muller, \emph{Some continuous monte carlo methods for the dirichlet
  problem}, The Annals of Mathematical Statistics (1956), 569--589.

\bibitem{Muller-Gronbach2002}
Thomas M{\"u}ller-Gronbach, \emph{The optimal uniform approximation of systems
  of stochastic differential equations}, aoap \textbf{12} (2002), no.~2,
  664--690 (en).

\bibitem{naeh90}
T.~Naeh, M.~M. K\l~osek, B.~J. Matkowsky, and Z.~Schuss, \emph{A direct
  approach to the exit problem}, SIAM J. Appl. Math. \textbf{50} (1990), no.~2,
  595--627. \MR{1043605}

\bibitem{neuenkirch2019adaptive}
Andreas Neuenkirch, Michaela Sz{\"o}lgyenyi, and Lukasz Szpruch, \emph{An
  adaptive euler--maruyama scheme for stochastic differential equations with
  discontinuous drift and its convergence analysis}, SIAM Journal on Numerical
  Analysis \textbf{57} (2019), no.~1, 378--403.

\bibitem{schuss80}
Zeev Schuss, \emph{Theory and applications of stochastic differential
  equations}, Wiley Series in Probability and Statistics, John Wiley \& Sons,
  Inc., New York, 1980. \MR{595164}

\bibitem{Verdugo2022}
Francesc Verdugo and Santiago Badia, \emph{The software design of gridap: A
  finite element package based on the julia {JIT} compiler}, Computer Physics
  Communications \textbf{276} (2022), 108341.

\bibitem{weinan2021applied}
E~Weinan, Tiejun Li, and Eric Vanden-Eijnden, \emph{Applied stochastic
  analysis}, vol. 199, American Mathematical Soc., 2021.

\bibitem{yaroslavtseva2022adaptive}
Larisa Yaroslavtseva, \emph{An adaptive strong order 1 method for sdes with
  discontinuous drift coefficient}, Journal of Mathematical Analysis and
  Applications \textbf{513} (2022), no.~2, 126180.

\end{thebibliography}
\end{document}